%% file: adelman.tex
\documentclass[12pt]{amsart}

\usepackage[pdfauthor   = {Sebastian\ Posur},
            pdftitle    = {Computable\ Free\ abelian\ categories},
            pdfsubject  = {},
            pdfkeywords = {computable abelian category},
            bookmarks=true,
            bookmarksopen=true,
            pagebackref=true,
            hyperindex=true,
            colorlinks=true,
            linkcolor=blue,
            citecolor=blue,
            filecolor=blue,
            urlcolor=blue,
            ]{hyperref}

\include{header}

\author{Sebastian Posur}
\thanks{This is a contribution to  Project-ID 286237555 – TRR 195 -- by the Deutsche Forschungsgemeinschaft (DFG, German Research Foundation).}
\address{Algebra and Representation Theory, RWTH Aachen University, Pontdriesch 10-16, 52062 Aachen, Germany}
\email{\href{mailto:Sebastian Posur <posur@art.rwth-aachen.de>}{posur@art.rwth-aachen.de}}

\begin{document}

\title[On free abelian categories for theorem proving]{On free abelian categories for theorem proving}

\begin{abstract}
We give a computational approach to theorem proving in homological algebra.
This approach is based on computations in the free abelian category of an additive category $\AC$.
We show that the free abelian category
is amenable to explicit computations whenever we can decide homotopy equations in $\AC$.
As some consequences of our investigations, we recover Dowker's explicit formula for the connecting homomorphism $\partial$
in the snake lemma, we find a universal sense in which $\partial$ is unique, and we give a refined version of the 5-lemma.
\end{abstract}

\maketitle
\setcounter{tocdepth}{4}
\tableofcontents

\input{adelman_category.tex}

\input{adelman.bbl}

\end{document}

%% file: header.tex
\usepackage[utf8]{inputenc} 
\usepackage[T1]{fontenc}
\usepackage{a4wide}
\usepackage{lmodern}
\usepackage[english]{babel}
\usepackage{mathrsfs}
\usepackage{etex}
\usepackage{mathtools}
\usepackage{latexsym}
\usepackage{amssymb}
\usepackage{amsthm}
\usepackage{amsmath}
\usepackage{caption}
\usepackage{mathabx}

\usepackage{colortbl}

\usepackage[all]{xy}
\usepackage{verbatim}
\usepackage{listings}
\usepackage{fancyvrb}

\usepackage{graphicx}

\usepackage[dvipsnames]{xcolor}
\usepackage{accents} 
\usepackage{enumerate}
\usepackage{wrapfig}
\usepackage{tikz}
\usepackage{tikz-cd}
\usetikzlibrary{automata,shapes,arrows,matrix,backgrounds,positioning,plotmarks,calc,patterns,matrix,decorations.pathreplacing,decorations.pathmorphing,decorations.text,decorations.markings}
\usepackage[colorinlistoftodos,shadow]{todonotes}

\usepackage{multirow}
\usepackage{mdwlist}

\usepackage{stmaryrd}
\usepackage{mathdots} 

\usepackage{fancyvrb}

\usepackage[toc,page]{appendix}
\usepackage{float}

\usepackage{extarrows}
\usepackage{pdflscape}
\usepackage{rotating}
\usepackage{etex}
\newtheoremstyle{mytheoremstyle} 
    {5pt}                    
    {5pt}                    
    {\itshape}                   
    {\parindent}                           
    {\bf}                   
    {.}                          
    {.5em}                       
    {}  

\theoremstyle{mytheoremstyle}

\newtheorem{theorem}{Theorem}[section]

\newtheorem{lemma}[theorem]{Lemma}

\newtheorem{corollary}[theorem]{Corollary}

\newtheoremstyle{mytdefintionstyle} 
    {5pt}                    
    {5pt}                    
    {\rm}                   
    {\parindent}                           
    {\bf}                   
    {.}                          
    {.5em}                       
    {}  

\theoremstyle{remark}
\newtheorem{remark}[theorem]{Remark}

\theoremstyle{mytdefintionstyle}
\newtheorem{definition}[theorem]{Definition}

\newtheorem{example}[theorem]{Example}

\newtheorem{computation}[theorem]{Computation}
\newtheorem{construction}[theorem]{Construction}

\newtheoremstyle{exmp_contd} 
{\topsep} {\topsep}%
{\upshape}
{}
{\bfseries}
{}
{ }
{\thmname{#1}\,\thmnumber{ #2}\thmnote{#3}\enspace(continued)}

\theoremstyle{exmp_contd}

\usepackage{xspace}
\input{pre.tex}

\tikzset{round left paren/.style={ncbar=0.5cm,out=120,in=-120}}
\tikzset{round right paren/.style={ncbar=0.5cm,out=60,in=-60}}
\newcolumntype{C}[1]{>{\centering\arraybackslash$}p{#1}<{$}}
\newlength{\mycolwd}
\usepackage{array, xcolor}
\definecolor{lightgray}{gray}{0.8}
\newcolumntype{L}{>{\raggedleft}p{0.28\textwidth}}
\newcolumntype{R}{p{0.8\textwidth}}

\definecolor{ctcolor}{gray}{0.95}

\definecolor{ctucolor}{gray}{0.85}

\makeatletter
\newcommand{\thickhline}{%
    \noalign {\ifnum 0=`}\fi \hrule height 1pt
    \futurelet \reserved@a \@xhline
}
\newcolumntype{"}{@{\hskip\tabcolsep\vrule width 1pt\hskip\tabcolsep}}
\makeatother

\usepackage{enumitem}
\newlist{theoremenumerate}{enumerate}{1}
\setlist[theoremenumerate]{label=(\arabic{theoremenumeratei}), ref=\thetheorem.(\arabic{theoremenumeratei}),noitemsep}

%% file: pre.tex
\definecolor{ExQ}{HTML}{0000FF}
\definecolor{Dec}{HTML}{E07B00}

\newcommand{\CapPkg}{\textsc{Cap}\xspace}

\newcommand{\EmbAdelFunctor}{\mathrm{Emb}}
\newcommand{\ObjAdel}[5]{ ({#1} \xrightarrow{#2} {{#3}} \xrightarrow{#4} {#5}) }
\newcommand{\EmbAdel}[1]{ \ObjAdel{0}{}{#1}{}{0} }
\newcommand{\MorAdel}[1]{ {   {\{#1\}}} }
\newcommand{\snakecolor}{blue}
\newcommand{\AsCat}{\mathcal{C}}
\newcommand{\Nzero}{\mathbb{Z}_{\geq 0}}
\newcommand{\AC}{\mathbf{A}}
\newcommand{\BC}{\mathbf{B}}

\newcommand{\CC}{\mathbf{C}}

\newcommand{\FC}{\mathbf{F}}

\newcommand{\pmatrow}[2]{ \begin{pmatrix}{#1} & {#2} \end{pmatrix} }
\newcommand{\pmatrowthree}[3]{ \begin{pmatrix}{#1} & {#2} & {#3} \end{pmatrix} }
\newcommand{\pmattwobytwo}[4]{ \begin{pmatrix}{#1} & {#2} \\{#3} & {#4} \end{pmatrix} }
\newcommand{\pmattwobythree}[6]{ \begin{pmatrix}{#1} & {#2} & {#3} \\{#4} & {#5} & {#6} \end{pmatrix} }
\newcommand{\pmatthreebytwo}[6]{ \begin{pmatrix}{#1} & {#2} \\ {#3} &{#4} \\ {#5} & {#6} \end{pmatrix} }
\newcommand{\pmatthreebythree}[9]{ \begin{pmatrix}{#1} & {#2} & {#3} \\{#4} & {#5} & {#6} \\{#7} & {#8} & {#9}\end{pmatrix} }
\newcommand{\pmatcol}[2]{ \begin{pmatrix}{#1} \\ {#2} \end{pmatrix} }

\newcommand{\pmattwobytwos}[4]{ \text{\tiny$\begin{pmatrix}{#1} & {#2} \\{#3} & {#4} \end{pmatrix}$ } }

\newcommand{\Z}{\mathbb{Z}}

\newcommand{\Q}{\mathbb{Q}}

\newcommand{\CH}{\mathrm{H}}

\newcommand{\op}{\mathrm{op}}

\newcommand{\Adel}{\mathrm{Adel}}

\newcommand{\id}{\mathrm{id}}

\DeclareMathOperator{\Aut}{\mathrm{Aut}}

\DeclareMathOperator{\Hom}{\mathrm{Hom}}

\DeclareMathOperator{\kernel}{\mathrm{ker}}

\DeclareMathOperator{\cokernel}{\mathrm{coker}}

\DeclareMathOperator{\image}{\mathrm{im}}

\newcommand{\Modl}{\text{-}\mathrm{Mod}}

\newcommand{\modl}{\text{-}\mathrm{mod}}

\newcommand{\CokernelProjection}{\mathrm{CokernelProjection}}

%% file: adelman_category.tex
Being able to compute explicitly within a universal mathematical object
can be interpreted as theorem proving.
For example, $U \coloneq \Q[a,b]/\langle a^3b + a, b^3a + b, a^2b^2 + 1 \rangle$
as a $\Q$-algebra is universal w.r.t.\ the property that there are two elements $a, b$
satisfying $a^3b + a = b^3a + b = a^2b^2 + 1 = 0$.
Due to Gröbner basis techniques \cite{CLO}, it is easy to compute within $U$ and to see that $U \cong 0$.
This computation in turn can be interpreted as a proof of the simple theorem that any commutative $\Q$-algebra with two elements satisfying the above relations
is already trivial.

A far more sophisticated example of a universal mathematical object is given by the \emph{free abelian category}.
Every additive category $\AC$ admits a universal additive functor into an abelian category
$\AC \xrightarrow{E} \FC$. Here, universal means that if we are given another additive functor
$\AC \xrightarrow{F} \BC$ into an abelian category, then there exists an exact functor (unique up to natural isomorphism)
$\FC \xrightarrow{\widehat{F}} \BC$ such that $\widehat{F} \circ E \simeq F$, i.e., such that the following diagram commutes (up to natural isomorphism):
\begin{center}
  \begin{tikzpicture}[label/.style={postaction={
          decorate,
          decoration={markings, mark=at position .5 with \node #1;}}}, baseline = (A),
          mylabel/.style={thick, draw=none, align=center, minimum width=0.5cm, minimum height=0.5cm,fill=white}]
        \coordinate (r) at (5,0);
        \coordinate (u) at (0,2);
        \node (A) {$\AC$};
        \node (B) at ($(A)+(r)$) {$\FC$};
        \node (C) at ($(A) + (r) - (u)$) {$ \BC $};
        \draw[->,thick] (A) --node[above]{$E$} (B);
        \draw[->,thick] (A) --node[below]{$F$} (C);
        \draw[->,thick, dashed] (B) --node[right]{$\widehat{F}$} (C);
  \end{tikzpicture}    
\end{center}
The abelian category $\FC$ is known as the free abelian category of $\AC$.

The existence of free abelian categories was first proven by Peter Freyd \cite[Theorem 4.1]{FreydRep}.
An easily graspable construction of free abelian categories arises from the theory of finitely presented functors:
a covariant functor from $\AC$ to the category of abelian groups is finitely presented if
it arises as the cokernel of a natural transformation between representable functors.
If $\AC\modl$ denotes the category of finitely presented functors,
then $\FC \simeq (\AC\modl)\modl$. This point of view is fruitfully applied in the context of model theory \cite{PrestPSL},
representation theory \cite{AusFun, Herzog08}, or in the determination of Diophantine sets \cite{HID14}.

Since categories of finitely presented functors are amenable to explicit computations \cite{PosFreyd} via so-called
Freyd categories \cite{BelFredCats},
it is natural to expect that we can perform explicit computations within free abelian categories.
Moreover, as a universal mathematical object, it is natural to expect that computing within free abelian categories
amounts to theorem proving. 
Thus, the goal of this paper is to show 
\begin{enumerate}
  \item how to compute explicitly with free abelian categories,
  \item how to apply this knowledge to theorem proving.
\end{enumerate}

The most explicit construction of free abelian categories is due to Murray Adelman \cite{Adelman}.
In the first section of this paper, we recall his construction $\Adel( \AC )$ and call it the \emph{Adelman category} of $\AC$.
We prove that within a constructive context like Bishop's constructive mathematics (see \cite{MRRConstructiveAlgebra}),
$\Adel( \AC )$ is a computable abelian category if we can decide homotopy equations in $\AC$ (Theorem \ref{theorem:Adelman_computable}).
In particular, this is the case for additive categories generated by acyclic quivers with relations (Example \ref{example:notation}),
which enables us to calculate in free abelian categories associated to several diagrams which occur in the premises of classical
homological lemmata.

In the second section of this paper, we apply the universal property and the computability of Adelman categories to theorem proving.
We do this by providing \emph{universal instances} of classical lemmata.
By a universal instance of a lemma $\mathcal{L}$, we mean an instance of $\mathcal{L}$
with the property that if $\mathcal{L}$ holds for the universal instance, it holds for all instances.

In Subsection \ref{subsection:up_of_adel}, we recall the universal property of the Adelman category.

In Subsection \ref{subsection:universal_snake},
we give the \emph{universal instance of the snake lemma} (Figure \ref{fig:universal_instance_snake}).
We recover an explicit formula of the connecting homomorphism (Remark \ref{remark:concrete_formula}) found by Dowker in \cite{Dow66}.
Moreover, we prove a universal uniqueness property of the connecting homomorphism (Lemma \ref{lemma:universal_uniqueness} and Remark \ref{remark:interpret_universal_uniqueness}).

In Subsection \ref{subsection:five_lemma}, we deal with the 5-lemma.
In contrast to the snake lemma, it is not clear how a universal instance of the 5-lemma can be realized within an Adelman category.
Thus, we first state a refinement of the 5-lemma \ref{lemma:five_refined}
(that actually generalizes the classical 5-lemma),
and afterwards we give the universal instance of this refinement (Lemma \ref{lemma:universal_instace_5lem}).

We end with a conclusion and outlook in Section \ref{section:outlook}.

\section{Computing in free abelian categories}

Throughout this section, $\AC$ denotes an additive category.

\subsection{Adelman categories}

\begin{definition}
  By a \textbf{composable pair} in $\AC$ we mean a pair of morphisms of the form
  \begin{center}
       \begin{tikzpicture}[label/.style={postaction={
            decorate,
            decoration={markings, mark=at position .5 with \node #1;}},
            mylabel/.style={thick, draw=none, align=center, minimum width=0.5cm, minimum height=0.5cm,fill=white}}]
            \coordinate (r) at (3,0);
            \node (A) {$a$};
            \node (B) at ($(A)+(r)$) {$b$};
            \node (C) at ($(B) + (r)$) {$c$.};
            \draw[->,thick] (A) to node[above]{$\alpha$} (B);
            \draw[->,thick] (B) to node[above]{$\beta$}(C);
      \end{tikzpicture}
  \end{center}
  The category of composable pairs is given by
  the functor category $\AC^{\Delta}$,
  where $\Delta$ is given by the diagram $\bullet \rightarrow \bullet \rightarrow \bullet$.
\end{definition}

\begin{remark}\label{remark:homologies}
Although we do \emph{not necessarily} have $\alpha \cdot \beta = 0$, we can nevertheless speak of the \textbf{homology} of a composable pair in the case when $\AC$ is abelian:
\[
 \CH( a \xrightarrow{\alpha} b \xrightarrow{\beta} c ) := 
 \frac{\kernel( \beta )}{\image( \alpha )} := 
 \frac{\kernel( \beta ) + \image( \alpha )}{\image( \alpha )} \simeq \frac{\kernel( \beta )}{\image( \alpha ) \cap \kernel( \beta )}
\]
or diagrammatically
\begin{center}
     \begin{tikzpicture}[label/.style={postaction={
          decorate,
          decoration={markings, mark=at position .5 with \node #1;}},
          mylabel/.style={thick, draw=none, align=center, minimum width=0.5cm, minimum height=0.5cm,fill=white}}]
          \coordinate (r) at (3,0);
          \coordinate (d) at (0,-2);
          \node (A) {$a$};
          \node (B) at ($(A)+(r)$) {$b$};
          \node (C) at ($(B) + (r)$) {$c$};
          \node (ker) at ($(A) + (d) + 0.5*(r)$) {$\kernel( \beta )$};
          \node (coker) at ($(ker) + (r)$) {$\cokernel( \alpha )$};
          \node (H) at ($(B) + 2*(d)$) {$\CH( a \stackrel{\alpha}{\longrightarrow} b \stackrel{\beta}{\longrightarrow} c )$};
          
          \draw[->,thick] (A) to node[above]{$\alpha$} (B);
          \draw[->,thick] (B) to node[above]{$\beta$}(C);
          \draw[right hook->,thick] (ker) to (B);
          \draw[->,thick] (ker) to (coker);
          \draw[->>,thick] (B) to (coker);
          \draw[right hook->,thick] (H) to (coker);
          \draw[->>,thick] (ker) to (H);
    \end{tikzpicture}
\end{center}
In this case, the functoriality of kernels, cokernels, and images implies that $\CH$ defines a functor
\[
  \CH: \AC^{\Delta} \rightarrow \AC.
\]
\end{remark}

The idea of Adelman's construction of free abelian categories
is to formally adjoin homologies in 
the sense of Remark \ref{remark:homologies} to a given additive category $\AC$ (that does not necessarily have kernels or cokernels).

\begin{construction}
 Let $\AC$ be an additive category.
 The \textbf{Adelman category} $\Adel( \AC )$ of $\AC$ is given by the following data:
 \begin{enumerate}
  \item Objects in $\Adel( \AC )$ are composable pairs in $\AC$:
  \begin{center}
     \begin{tikzpicture}[label/.style={postaction={
          decorate,
          decoration={markings, mark=at position .5 with \node #1;}},
          mylabel/.style={thick, draw=none, align=center, minimum width=0.5cm, minimum height=0.5cm,fill=white}}]
          \coordinate (r) at (3,0);
          \node (A) {$(r_a$};
          \node (B) at ($(A)+(r)$) {$a$};
          \node (C) at ($(B) + (r)$) {$c_a)$.};
          \draw[->,thick] (A) to node[above]{$\rho_a$} (B);
          \draw[->,thick] (B) to node[above]{$\gamma_a$}(C);
    \end{tikzpicture}
  \end{center}
  Note that the objects $r_a, c_a$ and morphisms $\rho_a, \gamma_a$ do not formally depend\footnote{We like to write
  objects in $\Adel( \AC )$ in this way in order to stress the interpretation of $\rho_a$ as ``imposing relations'' on $a$ and of $\gamma_a$ as ``imposing corelations'' on $a$.} on the object $a$ in the middle.
  We refer to $\rho_a$ as the \textbf{relation morphism} and to $\gamma_a$ as the \textbf{corelation morphism}.
  Whenever we regard a composable pair as an object in an Adelman category, we will write it in round brackets as depicted above.
  \item A morphism in $\Adel( \AC )$ from
  $\ObjAdel{r_a}{\rho_a}{a}{\gamma_a}{c_a}$
  to
  $\ObjAdel{r_b}{\rho_b}{b}{\gamma_b}{c_b}$
  is given by a morphism $\alpha: a \rightarrow b$ such that there exist morphisms
  $\omega_{\alpha}: r_a \longrightarrow r_b$ and $\psi_{\alpha}: c_a \longrightarrow c_b$
  such that the diagram 
  \begin{center}
     \begin{tikzpicture}[label/.style={postaction={
          decorate,
          decoration={markings, mark=at position .5 with \node #1;}},
          mylabel/.style={thick, draw=none, align=center, minimum width=0.5cm, minimum height=0.5cm,fill=white}}]
          \coordinate (r) at (3,0);
          \coordinate (d) at (0,-2);
          \node (A) {$r_a$};
          \node (B) at ($(A)+(r)$) {$a$};
          \node (C) at ($(B) + (r)$) {$c_a$};
          \node (A2) at ($(A)+(d)$) {$r_{b}$};
          \node (B2) at ($(B)+(d)$) {${b}$};
          \node (C2) at ($(C)+(d)$) {$c_{b}$};
          \draw[->,thick] (A) to node[above]{$\rho_a$} (B);
          \draw[->,thick] (B) to node[above]{$\gamma_a$}(C);
          \draw[->,thick] (A2) to node[above]{$\rho_b$} (B2);
          \draw[->,thick] (B2) to node[above]{$\gamma_b$}(C2);
          \draw[->,thick] (B) to node[left]{$\alpha$}(B2);
          \draw[->,thick,dotted] (A) to node[left]{$\omega_{\alpha}$}(A2);
          \draw[->,thick,dotted] (C) to node[right]{$\psi_{\alpha}$}(C2);
    \end{tikzpicture}
  \end{center}
  commutes. We denote such a morphism by $\{ \omega_{\alpha}, \alpha, \psi_{\alpha} \}$ or simply $\MorAdel{\alpha}$ and refer to 
  $\alpha$ as the \textbf{morphism datum}, to
  $\omega_{\alpha}$ as a \textbf{relation witness}\footnote{We will depict morphisms that we call witnesses by dotted arrows within diagrams.}
  and to $\psi_{\alpha}$ as a \textbf{corelation witness}.
  Moreover, we impose the following equivalence relation: we define two morphisms $\MorAdel{\alpha}$ and $\MorAdel{\alpha'}$ with the same source and range
  as \textbf{equal in $\Adel( \AC )$} if there exist morphisms $a \xrightarrow{\sigma_1} r_b$ and $c_a \xrightarrow{\sigma_2} b$ such that $\alpha - \alpha' = \sigma_1 \cdot \rho_b + \gamma_a \cdot \sigma_2$:
  \begin{center}
     \begin{tikzpicture}[label/.style={postaction={
          decorate,
          decoration={markings, mark=at position .5 with \node #1;}},
          mylabel/.style={thick, draw=black, align=center, minimum width=0.5cm, minimum height=0.5cm,fill=white}}]
          \coordinate (r) at (3,0);
          \coordinate (d) at (0,-2);
          \node (A) {};
          \node (B) at ($(A)+(r)$) {$a$};
          \node (C) at ($(B) + (r)$) {$c_a$};
          \node (A2) at ($(A)+(d)$) {$r_{b}$};
          \node (B2) at ($(B)+(d)$) {${b}$};
          \node (C2) at ($(C)+(d)$) {};
          \draw[->,thick] (B) to node[above]{$\gamma_a$}(C);
          \draw[->,thick] (A2) to node[above]{$\rho_b$} (B2);
          
          \draw[->,thick,label={[mylabel]{$\alpha - \alpha'$}}] (B) -- (B2);
          
          \draw[->,thick,dotted] (B) to node[above,xshift=-0.2em]{$\sigma_1$}(A2);
          \draw[->,thick,dotted] (C) to node[above,xshift=-0.2em]{$\sigma_2$}(B2);
    \end{tikzpicture}
  \end{center}
  We call any such pair $\sigma_1, \sigma_2$ a \textbf{witness pair} for
  the morphisms $\MorAdel{\alpha}$ and $\MorAdel{\alpha'}$ being equal.
 \end{enumerate}
\end{construction}

It is easy to check that the above construction gives rise to a well-defined category.
Moreover, $\Adel( \AC )$ can be seen as an additive quotient of the functor category $\AC^{\Delta}$,
where $\Delta$ is given by the diagram $\bullet \rightarrow \bullet \rightarrow \bullet$. In particular,
$\Adel( \AC )$ inherits its additive structure from $\AC^{\Delta}$, e.g., direct sums are built pointwise.

\begin{remark}[Duality]\label{remark:adel_duality}
 Sending an object 
 \begin{center}
     \begin{tikzpicture}[label/.style={postaction={
          decorate,
          decoration={markings, mark=at position .5 with \node #1;}},
          mylabel/.style={thick, draw=none, align=center, minimum width=0.5cm, minimum height=0.5cm,fill=white}}]
          \coordinate (r) at (3,0);
          \node (A) {$(r_a$};
          \node (B) at ($(A)+(r)$) {$a$};
          \node (C) at ($(B) + (r)$) {$c_a)$};
          \draw[->,thick] (A) to node[above]{$\rho_a$} (B);
          \draw[->,thick] (B) to node[above]{$\gamma_a$}(C);
    \end{tikzpicture}
 \end{center}
 in $\Adel( \AC )$ to
 \begin{center}
     \begin{tikzpicture}[label/.style={postaction={
          decorate,
          decoration={markings, mark=at position .5 with \node #1;}},
          mylabel/.style={thick, draw=none, align=center, minimum width=0.5cm, minimum height=0.5cm,fill=white}}]
          \coordinate (r) at (3,0);
          \node (A) {$(c_a$};
          \node (B) at ($(A)+(r)$) {$a$};
          \node (C) at ($(B) + (r)$) {$r_a)$};
          \draw[->,thick] (A) to node[above]{$\gamma_a^{\op}$} (B);
          \draw[->,thick] (B) to node[above]{$\rho_a^{\op}$}(C);
    \end{tikzpicture}
 \end{center}
 in $\Adel( \AC^{\op} )$ defines an anti-equivalence from $\Adel( \AC )$ to $\Adel( \AC^{\op} )$,
 i.e., an equivalence
 \[
  \Adel( \AC )^{\op} \simeq \Adel( \AC^{\op} ).
 \]
\end{remark}

\subsection{Kernels and cokernels in Adelman categories}

In this subsection, we recall the construction of kernels and cokernels in $\Adel( \AC )$.
Our presented construction differs slightly from that presented in \cite{Adelman},
since we tried to minimize the number of occurring minus signs and maximize the number of occurring zeros within entries of matrices.

\begin{construction}[Cokernels]\label{construction:cokernels}
Given a morphism
\[
  \ObjAdel{r_a}{\rho_a}{a}{\gamma_a}{c_a} \xrightarrow{\{ \omega_{\alpha}, \alpha, \psi_{\alpha} \}} \ObjAdel{r_b}{\rho_b}{b}{\gamma_b}{c_b}
\]
in $\Adel( \AC )$, the following diagram depicts how we can construct its cokernel projection
along with the morphism induced by its universal property:
  \begin{center}
    \begin{tikzpicture}[label/.style={postaction={
            decorate,
            decoration={markings, mark=at position .5 with \node #1;}}}, baseline = (A),
            mylabel/.style={thick, draw=none, align=center, minimum width=0.5cm, minimum height=0.5cm,fill=white}]
          \coordinate (r) at (5.7,0);
          \coordinate (u) at (0,3);
          \node (A) {$\ObjAdel{r_a}{\rho_a}{a}{\gamma_a}{c_a}$};
          \node (B) at ($(A)+(r)$) {$\ObjAdel{r_b}{\rho_b}{b}{\gamma_b}{c_b}$};
          \node (C) at ($(B) + 0.5*(r) + (u)$) 
          {$\ObjAdel{r_b \oplus a}{ \text{  \tiny$\pmattwobytwo{\rho_b}{0}{\alpha}{\gamma_a}$ } }{ b \oplus c_a }{  \text{ \tiny $\pmattwobytwo{\gamma_b}{0}{0}{\id_{c_a}}$ } }{  c_b \oplus c_a } $};
          \node (T) at ($(B) + 0.5*(r) - (u)$) {$\ObjAdel{r_t}{\rho_t}{t}{\gamma_t}{c_t}$.};
          \draw[->,thick] (A) --node[above]{$\{ \omega_{\alpha}, \alpha, \psi_{\alpha} \}$} (B);
          \draw[->,thick] (B) --node[left,xshift=-0.1em]{$\{ \omega_{\tau}, \tau, \psi_{\tau} \}$} (T);
          \draw[->,thick] (B) --node[left,yshift=0.6em,xshift=0.3em]{\tiny $\{ \pmatrow{1}{0}$, $\pmatrow{1}{0}$, $\pmatrow{1}{0} \}$} (C);
          \draw[->,thick, dashed] (C) --node[right]{\tiny $\{ \pmatcol{\omega_{\tau}}{\sigma_1}$, $\pmatcol{\tau}{-\sigma_2}$, $\pmatcol{\psi_{\tau}}{-\sigma_2 \cdot \gamma_t} \}$} (T);
          
          \draw[->,thick,dotted,in=180,out=-50] (A) to node[left,xshift=-1em]{{\tiny witness pair:} $\sigma_1, \sigma_2$} (T);
    \end{tikzpicture}    
  \end{center}
How to read this diagram:
the solid arrow pointing up right is the cokernel projection of $\{ \omega_{\alpha}, \alpha, \psi_{\alpha} \}$.
The solid arrow pointing down right is a test morphism $\{ \omega_{\tau}, \tau, \psi_{\tau} \}$ for the cokernel
with witness pair $\sigma_1: a \longrightarrow r_t$, $\sigma_2: c_a \longrightarrow t$
for the composition 
\[ \{ \omega_{\alpha}, \alpha, \psi_{\alpha} \} \cdot \{ \omega_{\tau}, \tau, \psi_{\tau} \} \] 
being zero, i.e., the equation
\begin{equation}\label{equation:wit_pair_cok}
 \alpha \cdot \tau = \sigma_1 \cdot \rho_t + \gamma_a \cdot \sigma_2
\end{equation}
holds.
The dashed arrow pointing down is the morphism induced by the universal property of the cokernel.
\end{construction}

\begin{remark}[Interpretation of Construction \ref{construction:cokernels}]\label{remark:cokernel_homology_interpretation}
 If $\AC$ is abelian, and if we interpret the depicted composable pairs of Construction \ref{construction:cokernels}
 as homologies in the sense of Remark \ref{remark:homologies}, then Construction \ref{construction:cokernels}
 can be seen as the exact sequence
 \begin{center}
    \begin{tikzpicture}[label/.style={postaction={
            decorate,
            decoration={markings, mark=at position .5 with \node #1;}}}, baseline = (A),
            mylabel/.style={thick, draw=none, align=center, minimum width=0.5cm, minimum height=0.5cm,fill=white}]
          \coordinate (r) at (3,0);
          \node (A) {$\frac{\kernel( \gamma_a )}{\image( \rho_a )}$};
          \node (B) at ($(A)+(r)$) {$\frac{\kernel( \gamma_b )}{\image( \rho_b )}$};
          \node (C) at ($(B) + 2*(r)$) 
          {$
          \frac{\kernel\text{  \tiny$\pmattwobytwo{\gamma_b}{0}{0}{\id_{c_a}}$ }}{\image\text{  \tiny$\pmattwobytwo{\rho_b}{0}{\alpha}{\gamma_a}$ }}
          \simeq
          \frac{ \kernel( \gamma_b ) }{ \image( \rho_b ) + \alpha( \kernel( \gamma_a ) ) }
          $};
          \node (D) at ($(C) + 2*(r)$) {$0$.};
          \draw[->,thick] (A) -- (B);
          \draw[->,thick] (B) -- (C);
          \draw[->,thick] (C) -- (D);
    \end{tikzpicture}    
  \end{center}
\end{remark}

\begin{proof}[Correctness of Construction \ref{construction:cokernels}]
Clearly, the cokernel projection is well-defined.
Its postcomposition with $\{ \omega_{\alpha}, \alpha, \psi_{\alpha} \}$ yields zero
with 
\[
 \pmatrow{0}{1}: a \longrightarrow r_b \oplus a, \hspace{2em} \pmatrow{0}{-1}: c_a \longrightarrow b \oplus c_a
\]
as a witness pair.
For the well-definedness of the induced morphism,
we use the well-definedness of $\{ \omega_{\tau}, \tau, \psi_{\tau} \}$
and \eqref{equation:wit_pair_cok}. Moreover, since $\pmatrow{1}{0} \pmatcol{\tau}{-\sigma_2} = \tau$, the triangle in the diagram commutes.

For the uniqueness of the induced morphism, it suffices to show that the cokernel projection is an epi.
Given a morphism $\{ \omega, \pmatcol{u_1}{u_2}, \psi \}$ from the cokernel object to another object
$
(r_u \stackrel{\rho_u}{\longrightarrow}u \stackrel{\gamma_u}{\longrightarrow} c_u)
$
such that its composition with the cokernel projection $\{ \pmatrow{1}{0}$, $\pmatrow{1}{0}$, $\pmatrow{1}{0} \}$ yields zero
with witness pair 
\[
 \sigma_3: b \longrightarrow r_u, \hspace{2em} \sigma_4: c_b \longrightarrow u,
\]
the morphism $\{ \omega, \pmatcol{u_1}{u_2}, \psi \}$ itself is already zero with witness pair
\[
 \pmatcol{\sigma_3}{0}: b \oplus c_a \longrightarrow r_u, \hspace{2em} \pmatcol{\sigma_4}{u_2}: c_b \oplus c_a \longrightarrow u.
\]
\end{proof}

Due to Remark \ref{remark:adel_duality}, the construction of kernels in $\Adel( \AC )$ can be performed dually,
and we spell it out explicitly for future reference:

\begin{construction}[Kernels]\label{construction:kernels}
  The following diagram can be read analogously to Construction \ref{construction:cokernels}:
\begin{center}
  \begin{tikzpicture}[label/.style={postaction={
          decorate,
          decoration={markings, mark=at position .5 with \node #1;}}}, baseline = (A),
          mylabel/.style={thick, draw=none, align=center, minimum width=0.5cm, minimum height=0.5cm,fill=white}]
        \coordinate (r) at (5.7,0);
        \coordinate (u) at (0,3);
        \node (A) {$\ObjAdel{r_b}{\rho_b}{b}{\gamma_b}{c_b}$};
        \node (B) at ($(A)+(r)$) {$ \ObjAdel{r_a}{\rho_a}{a}{\gamma_a}{c_a}$};
        \node (C) at ($(B) + 0.5*(r) + (u)$) 
        {$\ObjAdel{r_a \oplus r_b}{\pmattwobytwos{\rho_a}{0}{0}{\id_{r_b}}}{a \oplus r_b}{\pmattwobytwos{\gamma_a}{\alpha}{0}{\rho_b}}{c_a \oplus b}$};
        \node (T) at ($(B) + 0.5*(r) - (u)$) {$(r_t \stackrel{\rho_t}{\longrightarrow}t \stackrel{\gamma_t}{\longrightarrow} c_t)$.};
        \draw[->,thick] (B) --node[above]{$\{ \omega_{\alpha}, \alpha, \psi_{\alpha} \}$} (A);
        \draw[->,thick] (T) --node[left,xshift=-0.1em]{$\{ \omega_{\tau}, \tau, \psi_{\tau} \}$} (B);
        \draw[->,thick] (C) --node[left,yshift=0.6em,xshift=0.3em]{\tiny $\{ \pmatcol{1}{0}$, $\pmatcol{1}{0}$, $\pmatcol{1}{0} \}$} (B);
        \draw[->,thick, dashed] (T) --node[right]{\tiny $\{ \pmatrow{\omega_{\tau}}{-\rho_{t} \cdot \sigma_1},  \pmatrow{\tau}{-\sigma_1}, \pmatrow{\psi_{\tau}}{\sigma_2} \}$} (C);
        
        \draw[->,thick,dotted,in=-50,out=180] (T) to node[left,xshift=-1em]{{\tiny witness pair:} $\sigma_1, \sigma_2$} (A);
  \end{tikzpicture}    
\end{center}
\end{construction}

\subsection{Monos as kernels and epis as cokernels in Adelman categories}

In this subsection, we give an explicit construction in $\Adel( \AC )$ that identifies an epi with the cokernel of its kernel.

\begin{construction}[Epis as cokernels]
Given an epi 
\[
 \{ \omega_{\alpha}, \alpha, \psi_{\alpha} \}: 
 \ObjAdel{r_a}{\rho_a}{a}{\gamma_a}{c_a}
 \longrightarrow
 \ObjAdel{r_b}{\rho_b}{b}{\gamma_b}{c_b}
\]
in $\Adel(\AC)$, we will show that it is equal to the cokernel of its kernel
(considered as factor objects of the source).
First, since $\{ \omega_{\alpha}, \alpha, \psi_{\alpha} \}$ is an epi, its cokernel projection (see Construction \ref{construction:cokernels})
\begin{center}
    \begin{tikzpicture}[label/.style={postaction={
            decorate,
            decoration={markings, mark=at position .5 with \node #1;}}}, baseline = (A),
            mylabel/.style={thick, draw=none, align=center, minimum width=0.5cm, minimum height=0.5cm,fill=white}]
          \coordinate (r) at (9.65,0);
          \coordinate (u) at (0,0.25);
          \node (B) {$\ObjAdel{r_b}{\rho_b}{b}{\gamma_b}{c_b}$};
          \node (C) at ($(B) +(r) + (u)$) 
          {
          $\ObjAdel{r_b \oplus a}{ \text{  \tiny$\pmattwobytwo{\rho_b}{0}{\alpha}{\gamma_a}$ } }{ b \oplus c_a }{  \text{ \tiny $\pmattwobytwo{\gamma_b}{0}{0}{\id_{c_a}}$ } }{  c_b \oplus c_a } $
          };
          \node (C2) at ($(B) + 0.55*(r)$) {};
          \draw[->,thick] (B) --node[above]{\tiny $\{ \pmatrow{1}{0}$, $\pmatrow{1}{0}$, $\pmatrow{1}{0} \}$} (C2);
    \end{tikzpicture}    
  \end{center}
is zero. We let
\[
 \pmatrow{\sigma_7}{\sigma_8}: b \longrightarrow r_b \oplus a \hspace{5em}
 \pmatrow{\sigma_5}{\sigma_6}: c_b \longrightarrow b \oplus c_a
\]
denote a witness pair for the cokernel projection being zero, i.e., the equations
\begin{equation}\label{equation:wp_cok1}
 \id_b = \gamma_b \cdot \sigma_5 + \sigma_7 \cdot \rho_b + \sigma_8 \cdot \alpha
\end{equation}
and
\begin{equation}\label{equation:wp_cok2}
 _{b}0_{c_a} = \gamma_b \cdot \sigma_6 + \sigma_8 \cdot \gamma_a
\end{equation}
hold.
Using Construction \ref{construction:cokernels} and Construction \ref{construction:kernels} for computing the cokernel projection
of the kernel embedding of $\{ \omega_{\alpha}, \alpha, \psi_{\alpha} \}$ yields
the solid arrow pointing down right in the diagram
\begin{center}
    \begin{tikzpicture}[label/.style={postaction={
            decorate,
            decoration={markings, mark=at position .5 with \node #1;}}}, baseline = (A),
            mylabel/.style={thick, draw=none, align=center, minimum width=0.5cm, minimum height=0.5cm,fill=white}]
          \coordinate (r) at (7,0);
          \coordinate (u) at (0,3);
          \node (A) {$\ObjAdel{r_a}{\rho_a}{a}{\gamma_a}{c_a}$};
          \node (B) at ($(A)+(r)$) {$\ObjAdel{r_b}{\rho_b}{b}{\gamma_b}{c_b}$};
          \node (C) at ($(A) + (r) - (u)$) 
          {
          $
          \ObjAdel
          {r_a \oplus (a \oplus r_b)}
          { { \text{  \tiny$\pmatthreebythree{\rho_a}{0}{0}{\id_a}{\gamma_a}{\alpha}{0}{0}{\rho_b}$ } } }
          {a \oplus (c_a \oplus b)}
          { { \text{ \tiny $\pmatthreebythree{\gamma_a}{0}{0}{0}{\id_{c_a}}{0}{0}{0}{\id_{b}}$ } } }
          {c_a \oplus (c_a \oplus b)}
          $
          };
          \draw[->,thick] (A) --node[above]{$\{ \omega_{\alpha}, \alpha, \psi_{\alpha} \}$} (B);
          
          \draw[->,thick,shorten >=0.5cm] (A) --node[left,yshift=-0.3em]{\tiny $\{ \pmatrowthree{1}{0}{0}$, $\pmatrowthree{1}{0}{0}$, $\pmatrowthree{1}{0}{0} \}$} ($(C)$);
          
          \draw[->,thick, dashed] (B) --node[right]{\textbf{?}} ($(C)$);
    \end{tikzpicture}    
  \end{center}
If we are able to construct a dashed morphism rendering the triangle in the above diagram commutative, then we are done.
We make the following ansatz for the morphism datum:
\[
 \pmatrowthree{\sigma_8}{-\gamma_b \cdot \sigma_6}{-\gamma_b \cdot \sigma_5}: b \longrightarrow a \oplus (c_a \oplus b)
\]
\end{construction}
\begin{proof}[Correctness of the construction]
The diagram
\begin{center}
    \begin{tikzpicture}[label/.style={postaction={
            decorate,
            decoration={markings, mark=at position .5 with \node #1;}}}, baseline = (A),
            mylabel/.style={thick, draw=none, align=center, minimum width=0.5cm, minimum height=0.5cm,fill=white}]
          \coordinate (r) at (6,0);
          \coordinate (u) at (0,-4);
          
          \node (RB) {$r_b$};
          \node (B) at ($(RB) + (r)$) {$b$};
          \node (CB) at ($(B) + (r)$) {$c_b$};
          
          \node (X) at ($(RB) + (u)$) {$r_a \oplus (a \oplus r_b)$};
          \node (Y) at ($(X) + (r)$) {$a \oplus (c_a \oplus b)$};
          \node (Z) at ($(Y) + (r)$) {$c_a \oplus (c_a \oplus b)$};
          
          \draw[->,thick] (RB) --node[above]{$\rho_b$} (B);
          \draw[->,thick] (B) --node[above]{$\gamma_b$} (CB);
          
          \draw[->,thick] (X) --node[above]{{  \tiny$\pmatthreebythree{\rho_a}{0}{0}{\id_a}{\gamma_a}{\alpha}{0}{0}{\rho_b}$ }} (Y);
          \draw[->,thick] (Y) --node[above]{{ \tiny $\pmatthreebythree{\gamma_a}{0}{0}{0}{\id_{c_a}}{0}{0}{0}{\id_{b}}$ }} (Z);

          \draw[->,thick] (B) --node[mylabel]{$\pmatrowthree{\sigma_8}{-\gamma_b \cdot \sigma_6}{-\gamma_b \cdot \sigma_5}$} (Y);
          \draw[->,thick,dashed] (RB) --node[mylabel]{$\pmatrowthree{0}{\rho_b \cdot \sigma_8}{\rho_b \cdot \sigma_7 - \id_b  }$} (X);
          \draw[->,thick,dashed] (CB) --node[mylabel]{$\pmatrowthree{-\sigma_6}{-\sigma_6}{-\sigma_5}$} (Z);
    \end{tikzpicture}    
  \end{center}
commutes and displays relation and corelation witnesses for our morphism datum due to Equations \ref{equation:wp_cok1} and \ref{equation:wp_cok2}.
The diagram
\begin{center}
    \begin{tikzpicture}[label/.style={postaction={
            decorate,
            decoration={markings, mark=at position .5 with \node #1;}}}, baseline = (A),
            mylabel/.style={thick, draw=none, align=center, minimum width=0.5cm, minimum height=0.5cm,fill=white}]
          \coordinate (r) at (6,0);
          \coordinate (u) at (0,-4);
          
          \node (RB) {};
          \node (B) at ($(RB) + (r)$) {$a$};
          \node (CB) at ($(B) + (r)$) {$c_a$};
          
          \node (X) at ($(RB) + (u)$) {$r_a \oplus (a \oplus r_b)$};
          \node (Y) at ($(X) + (r)$) {$a \oplus (c_a \oplus b)$};

          \draw[->,thick] (B) --node[above]{$\gamma_a$} (CB);
          
          \draw[->,thick] (X) --node[above]{{  \tiny$\pmatthreebythree{\rho_a}{0}{0}{\id_a}{\gamma_a}{\alpha}{0}{0}{\rho_b}$ }} (Y);

          \draw[->,thick] (B) --node[mylabel]{\tiny $\pmatrowthree{\alpha \cdot \sigma_8 - \id_{A} }{- \alpha \cdot \gamma_b \cdot \sigma_6  }{- \alpha \cdot \gamma_b \cdot \sigma_5}$} (Y);
          
          \draw[->,thick, dashed, out = 180, in = 90] (B) to node[mylabel]
          {\tiny $\pmatrowthree{0}{\alpha \cdot \sigma_8 - \id_a}{\alpha \cdot \sigma_7}$}
          (X);
          
          \draw[->,thick, dashed, out = -90, in = 0] (CB) to node[mylabel]
          {\tiny $\pmatrowthree{0}{\id_{c_a}}{0}$}
          (Y);
          
    \end{tikzpicture}    
  \end{center}
displays a witness pair (again due to Equations \ref{equation:wp_cok1} and \ref{equation:wp_cok2}) proving the desired commutativity of the triangle.
\end{proof}

Due to Remark \ref{remark:adel_duality}, the construction for monos as kernels of cokernels in $\Adel( \AC )$ can be performed dually.

\begin{corollary}
 $\Adel( \AC )$ is an abelian category.
\end{corollary}

\subsection{Computability of Adelman categories}\label{subsection:computability_of_adel}

Due to their explicitness, Adelman categories are amenable to computations.
In a constructive context\footnote{Like Bishop's constructive mathematics, see \cite{MRRConstructiveAlgebra}.}, 
a category is realized by data types for its objects and morphisms,
and algorithms for composition and the construction of identities.
It is called \textbf{computable} if we have an algorithm deciding whether two morphisms are equal.
An explanation of a constructive approach to category theory can be found in \cite[Section 2]{PosFreyd},
with a detailed definition of a computable additive or abelian category in \cite[Appendix]{PosFreyd}.

\begin{definition}\label{definition:decidable_homotopy_equations}
  We say an additive category $\AC$ has \textbf{decidable homotopy equations} if
  it comes equipped with an algorithm whose input are the solid arrows of the diagram
  \begin{center}
    \begin{tikzpicture}[label/.style={postaction={
         decorate,
         decoration={markings, mark=at position .5 with \node #1;}},
         mylabel/.style={thick, draw=black, align=center, minimum width=0.5cm, minimum height=0.5cm,fill=white}}]
         \coordinate (r) at (3,0);
         \coordinate (d) at (0,-2);
         \node (A) {};
         \node (B) at ($(A)+(r)$) {$a$};
         \node (C) at ($(B) + (r)$) {$c$};
         \node (A2) at ($(A)+(d)$) {$d$};
         \node (B2) at ($(B)+(d)$) {${b}$};
         \node (C2) at ($(C)+(d)$) {};
         \draw[->,thick] (B) to node[above]{$\gamma$}(C);
         \draw[->,thick] (A2) to node[above]{$\beta$} (B2);
         
         \draw[->,thick] (B) to node[left]{$\alpha$} (B2);
         
         \draw[->,thick,dotted] (B) to node[above,xshift=-0.2em]{$\sigma_1$}(A2);
         \draw[->,thick,dotted] (C) to node[above,xshift=-0.2em]{$\sigma_2$}(B2);
   \end{tikzpicture}
 \end{center}
  in $\AC$ and whose output are the dotted arrows
  such that
  \[
    \alpha = \sigma_1 \cdot \beta + \gamma \cdot \sigma_2.
  \]
\end{definition}

\begin{theorem}\label{theorem:Adelman_computable}
  Let $\AC$ be an additive category.
  Then $\Adel( \AC )$ is computable abelian if $\AC$ has decidable homotopy equations.
\end{theorem}
\begin{proof}
  The constructions of kernels, cokernels, lifts along monos, colifts along epis in Adelman categories
that we have seen are all carried out on a formal level. Only the construction of witnesses requires possibly non-trivial algorithms
for the underlying additive category $\AC$, which is encapsulated in our Definition \ref{definition:decidable_homotopy_equations}.
\end{proof}

\begin{remark}
  We remark that being able to algorithmically decide equality of morphisms in an abelian category is quite powerful.
  Let $\BC$ be a computable abelian category and $b \in \BC$ an object:
  \begin{enumerate}
    \item $b \cong 0$ iff $\id_b = 0$,
  \end{enumerate}
  thus we can decide whether an object is zero, from which we have for a morphism $\beta: b \rightarrow c$
  the following algorithmic tests:
  \begin{enumerate}
    \setcounter{enumi}{1}
    \item $\beta$ is a mono iff $\kernel( \beta ) \cong 0$,
    \item $\beta$ is an epi iff $\cokernel( \beta ) \cong 0$,
    \item $\beta$ is an iso iff it is both a mono and an epi.
  \end{enumerate}
  Furthermore, we can check relations between subobjects given by monos $b' \stackrel{\iota_{b'}}{\hookrightarrow} b \stackrel{\iota_{b''}}{\hookleftarrow} b''$:
  \begin{enumerate}
    \setcounter{enumi}{4}
    \item $b' \leq b''$ iff $\iota_{b'} \cdot \CokernelProjection( \iota_{b''} ) = 0$,
    \item $b' = b''$ as subobjects iff both $b' \leq b''$ and $b'' \leq b'$,
  \end{enumerate}
  and likewise for factor objects.
\end{remark}

In this paper, we are mainly interested in the following class of examples of computable Adelman categories.

\begin{example}\label{example:notation}
  Suppose given a quiver $Q$ and a set of $\Z$-linear relations $R$ of its paths.
  We denote by $\AsCat(Q,R)$ the $\Z$-linear category whose objects are the nodes of $Q$,
  and whose morphism sets are $\Z$-linear combinations of paths of $Q$ modulo the relations provided by $R$.
  If $\CC$ denotes a $\Z$-linear category, the $\CC^{\oplus}$ denotes its additive closure, i.e.,
  objects in $\CC^{\oplus}$ are (possibly empty) tuples of objects in $\CC$,
  and a morphism between two such tuples $(C_1, \dots, C_m)$ and $(D_1, \dots, D_n)$
  is given by an $m \times n$ matrix $A$ whose $(i,j)$-th entry is a morphism $C_i \rightarrow D_j$ in $\CC$,
  where $m,n \in \Nzero$, $1 \leq i \leq m$, $1 \leq j \leq n$, $C_1, \dots, C_m, D_1, \dots, D_n \in \CC$.
  Whenever a tuple $(C_1, \dots, C_m)$ is regarded as an object in $\CC^{\oplus}$, we write $C_1 \oplus \dots \oplus C_m$.
  
  Now, if the quiver $Q$ is finite and acyclic, then the homomorphism sets in $\AsCat(Q,R)$ are all finitely presented abelian groups.
  Solving the equation of Definition \ref{definition:decidable_homotopy_equations} in $\AsCat(Q,R)^{\oplus}$ then
  boils down to solving linear systems over finitely presented abelian groups (see \cite[Corollary 6.11]{PosFreyd}), 
  which in turn boils down to the computation of Hermite normal forms of matrices over $\Z$.
  It follows from Theorem \ref{theorem:Adelman_computable} that $\Adel( \AsCat(Q,R)^{\oplus} )$ is computable abelian in this case.
\end{example}

\section{Universal instances of homological lemmata}\label{section:applications}

In this section, we show how the universal property of Adelman categories (see Subsection \ref{subsection:up_of_adel})
together with the computability of Adelman categories (see Subsection \ref{subsection:computability_of_adel})
can be employed to prove a homological lemma $\mathcal{L}$ computationally.
We make use of the following strategy:
the premise of $\mathcal{L}$ is typically given by a diagram with certain properties situated in an arbitrary abelian category $\BC$
(see the snake lemma \ref{lemma:snake} for an example).
\begin{enumerate}
  \item 
        We construct a specific additive category $\AC$
        such that the premise of $\mathcal{L}$
        is encoded by an additive functor $\AC \rightarrow \BC$.
  \item We prove the conclusion of the homological lemma for the specific premise encoded 
  by the canonical embedding functor $\EmbAdelFunctor: \AC \rightarrow \Adel( \AC )$ (see Remark \ref{remark:can_emb}).
  We call it the \textbf{universal instance} of $\mathcal{L}$.
  \item The universal property of Adelman categories implies that $\mathcal{L}$ holds in general.
\end{enumerate}

In Subsection \ref{subsection:universal_snake}, we apply this strategy to the snake lemma.
In Subsection \ref{subsection:five_lemma}, we apply this strategy to (a refinement of) the 5-lemma.

\subsection{Universal property of Adelman categories}\label{subsection:up_of_adel}

\begin{remark}\label{remark:can_emb}
  An additive category $\AC$ embeds fully into its Adelman category $\Adel( \AC )$ via
  \[
    \EmbAdelFunctor: \AC \rightarrow \Adel( \AC ): a \mapsto \EmbAdel{a}
  \]
\end{remark}

\begin{theorem}[Universal property of the Adelman category, {\cite[Theorem 1.14]{Adelman}}]\label{theorem:up_of_adel}
 Let $\AC$ be an additive category, let $\BC$ be an abelian category, and let $F: \AC \rightarrow \BC$ be an additive functor.
 There exists a unique (up to natural isomorphism) exact functor 
 $\widehat{F}: \Adel( \AC ) \longrightarrow \BC$
 such that
 \begin{center}
    \begin{tikzpicture}[label/.style={postaction={
            decorate,
            decoration={markings, mark=at position .5 with \node #1;}}}, baseline = (A),
            mylabel/.style={thick, draw=none, align=center, minimum width=0.5cm, minimum height=0.5cm,fill=white}]
          \coordinate (r) at (5,0);
          \coordinate (u) at (0,3);
          \node (A) {$\AC$};
          \node (B) at ($(A)+(r)$) {$\Adel( \AC )$};
          \node (C) at ($(A) + (r) - (u)$) {$ \BC $};
          \draw[->,thick] (A) --node[above]{$\EmbAdelFunctor$} (B);
          
          \draw[->,thick] (A) --node[below]{$F$} (C);
          
          \draw[->,thick, dashed] (B) --node[right]{$\widehat{F}$} (C);
    \end{tikzpicture}    
  \end{center}
 commutes up to natural isomorphism.
\end{theorem}

\begin{remark}\label{remark:explicit_up_induced_fun}
From the discussion in \cite{Adelman} preceding Theorem 1.14,
it follows that such an $\widehat{F}$ can be described explicitly:
\[
  \widehat{F}: \ObjAdel{r_a}{\rho_a}{a}{\gamma_a}{c_a} \mapsto \CH\ObjAdel{Fr_a}{F\rho_a}{Fa}{F\gamma_a}{Fc_a}
\]
where $\CH$ denotes the homology as described in Remark \ref{remark:homologies}.
\end{remark}

\begin{remark}
  Let $a \xrightarrow{\alpha} b \xrightarrow{\beta} c$ be a composable pair in $\AC$.
  Then
  \[
    \CH\big( \EmbAdel{ a } \xrightarrow{\MorAdel{\alpha}} \EmbAdel{b} \xrightarrow{\MorAdel{\beta}} \EmbAdel{c} \big) \cong \ObjAdel{a}{\alpha}{b}{\beta}{c}
  \]
  in $\Adel( \AC )$. This either follows from the universal property of Adelman categories applied to the embedding functor $\EmbAdelFunctor$ itself,
  or from a direct computation.
\end{remark}

\subsection{The universal instance of the snake lemma}\label{subsection:universal_snake}

First, we recall the statement of the snake lemma.

\begin{lemma}[Snake lemma]\label{lemma:snake}
Suppose given the following commutative diagram with exact rows in an abelian category:

\begin{center}
  \begin{tikzpicture}[label/.style={postaction={
    decorate,
    decoration={markings, mark=at position .5 with \node #1;}},
    mylabel/.style={thick, draw=none, align=center, minimum width=0.5cm, minimum height=0.5cm,fill=white}}, scale = 0.7]
        \coordinate (r) at (2.5,0);
        \coordinate (d) at (0,-2);
        
        {
        \node (V1) {$a$};
        \node (V2) at ($(V1) + (r)$) {$b$};
        \node (V3) at ($(V2) + (d)$) {$c$};
        \node (V4) at ($(V3) + (r)$) {$d$};

        \draw[->,thick] (V1) to node[above]{$\alpha$} (V2);
        \draw[->,thick] (V2) to node[left]{$\beta$} (V3);
        \draw[->,thick] (V3) to node[above]{$\gamma$} (V4);
        }
        
        {
        \node (Coka) at ($(V2) + (r)$) {$e$};
        \draw[->,thick] (V2) to (Coka);
        }
        
        {
        \draw[->,thick] (Coka) to node[right]{$\epsilon$} (V4);
        }

        {
        \node (Kerc) at ($(V3) - (r)$) {$f$};
        \draw[->,thick] (Kerc) to  (V3);
        }
        
        {
          \node (Zu) at ($(Coka) + (r)$) {$0$};
          \node (Zd) at ($(Kerc) - (r)$) {$0$};
          \draw[->,thick] (Coka) to (Zu);
          \draw[->,thick] (Zd) to (Kerc);
        }

        {
        \draw[->,thick] (V1) to node[left]{$\delta$} (Kerc);
        }
        
  \end{tikzpicture}
\end{center}

Then we have a morphism $\partial: \kernel(\epsilon) \rightarrow \cokernel(\delta)$
(called the \textbf{connecting homomorphism}) and an exact sequence (called the \textbf{snake}) of the following form:
    
\begin{center}
   \begin{tikzpicture}[label/.style={postaction={
     decorate,
     decoration={markings, mark=at position .5 with \node #1;}},
     mylabel/.style={thick, draw=none, align=center, minimum width=0.5cm, minimum height=0.5cm,fill=white}}, scale = 1]
         \coordinate (r) at (2.5,0);
         \coordinate (d) at (0,-2);
         
         {
         \node (Kere) {$\kernel( \delta )$};
         \node (Kerb) at ($(Kere) + (r)$) {$\kernel( \beta )$};
         \node (Kerd) at ($(Kerb) + (r)$) {$\kernel( \epsilon )$};
         \node (Coke) at ($(Kerd) + (r)$) {$\cokernel( \delta )$};
         \node (Cokb) at ($(Coke) + (r)$) {$\cokernel( \beta )$};
         \node (Cokd) at ($(Cokb) + (r)$) {$\cokernel( \epsilon )$};

         \draw[->,thick] (Kere) to (Kerb);
         \draw[->,thick] (Kerb) to (Kerd);
         \draw[->,thick] (Kerd) to node[above]{$\partial$} (Coke);
         \draw[->,thick] (Coke) to (Cokb);
         \draw[->,thick] (Cokb) to (Cokd);
         }
         
   \end{tikzpicture}
\end{center}
\end{lemma}

Note that in the premise of Lemma \ref{lemma:snake} we always have $\alpha \cdot \beta \cdot \gamma = 0$.
Conversely, if we are given three consecutive morphisms $\alpha$, $\beta$, $\gamma$ whose composition yields zero,
the diagram in the premise of Lemma \ref{lemma:snake} can be reconstructed up to isomorphism. Thus, we may rephrase the snake lemma
as follows:

\begin{lemma}[Snake lemma, rephrased]\label{lemma:snake_rephrased}
  Suppose given three morphisms in an abelian category
  $a \xrightarrow{\alpha} b \xrightarrow{\beta} c \xrightarrow{\gamma} d$ 
  such that $\alpha \cdot \beta \cdot \gamma = 0$.
  Then $\alpha, \beta, \gamma$ fit into a commutative diagram with exact rows
  \begin{center}
    \begin{tikzpicture}[label/.style={postaction={
      decorate,
      decoration={markings, mark=at position .5 with \node #1;}},
      mylabel/.style={thick, draw=none, align=center, minimum width=0.5cm, minimum height=0.5cm,fill=white}}, scale = 0.7]
          \coordinate (r) at (2.5,0);
          \coordinate (d) at (0,-2);
          
          {
          \node (V1) {$a$};
          \node (V2) at ($(V1) + (r)$) {$b$};
          \node (V3) at ($(V2) + (d)$) {$c$};
          \node (V4) at ($(V3) + (r)$) {$d$};

          \draw[->,thick] (V1) to node[above]{$\alpha$} (V2);
          \draw[->,thick] (V2) to node[left]{$\beta$} (V3);
          \draw[->,thick] (V3) to node[above]{$\gamma$} (V4);
          }
          
          {
          \node (Coka) at ($(V2) + (r)$) {$\cokernel(\alpha)$};
          \draw[->,thick] (V2) to (Coka);
          }
          
          {
          \draw[->,thick] (Coka) to node[right]{$\epsilon$} (V4);
          }

          {
          \node (Kerc) at ($(V3) - (r)$) {$\kernel(\gamma)$};
          \draw[->,thick] (Kerc) to  (V3);
          }
          
          {
            \node (Zu) at ($(Coka) + (r)$) {$0$};
            \node (Zd) at ($(Kerc) - (r)$) {$0$};
            \draw[->,thick] (Coka) to (Zu);
            \draw[->,thick] (Zd) to (Kerc);
          }

          {
          \draw[->,thick] (V1) to node[left]{$\delta$} (Kerc);
          }
          
    \end{tikzpicture}
  \end{center}
  and we have an exact sequence
  \begin{center}
    \begin{tikzpicture}[label/.style={postaction={
      decorate,
      decoration={markings, mark=at position .5 with \node #1;}},
      mylabel/.style={thick, draw=none, align=center, minimum width=0.5cm, minimum height=0.5cm,fill=white}}, scale = 1]
          \coordinate (r) at (2.5,0);
          \coordinate (d) at (0,-2);
          
          {
          \node (Kere) {$\kernel( \delta )$};
          \node (Kerb) at ($(Kere) + (r)$) {$\kernel( \beta )$};
          \node (Kerd) at ($(Kerb) + (r)$) {$\kernel( \epsilon )$};
          \node (Coke) at ($(Kerd) + (r)$) {$\cokernel( \delta )$};
          \node (Cokb) at ($(Coke) + (r)$) {$\cokernel( \beta )$};
          \node (Cokd) at ($(Cokb) + (r)$) {$\cokernel( \epsilon )$};

          \draw[->,thick] (Kere) to (Kerb);
          \draw[->,thick] (Kerb) to (Kerd);
          \draw[->,thick] (Kerd) to node[above]{$\partial$} (Coke);
          \draw[->,thick] (Coke) to (Cokb);
          \draw[->,thick] (Cokb) to (Cokd);
          }
          
    \end{tikzpicture}
 \end{center}
  \end{lemma}

Note that the premise of Lemma \ref{lemma:snake_rephrased}
can now be encoded as an additive functor:
if $Q$ denotes the quiver given by $a \xrightarrow{\alpha} b \xrightarrow{\beta} c \xrightarrow{\gamma} d$
with relation $R$ given by  $\alpha \cdot \beta \cdot \gamma = 0$,
then an additive functor from $\AsCat( Q, R )^{\oplus}$ (see Example \ref{example:notation} for the notation) to an abelian category $\BC$
corresponds up to isomorphism to a triple of consecutive morphisms in $\BC$ whose composition is zero,
simply by evaluating the functor at $\alpha, \beta, \gamma$.
Thus, we are ready to state and prove the universal instance of the (rephrased) snake lemma.

\begin{lemma}[The universal instance of the snake lemma]\label{lemma:universal_snake_lemma}
  Let $Q$ denote the quiver given by
  \[ a \xrightarrow{\alpha} b \xrightarrow{\beta} c \xrightarrow{\gamma} d \]
  and let $R$ be given by the single relation
  \[
    \alpha \cdot \beta \cdot \gamma = 0.
  \]
  
  Then the diagram in $\Adel( \AsCat( Q, R )^{\oplus} )$ depicted in Figure \ref{fig:universal_instance_snake} commutes, has exact rows, exact columns,
  and moreover, the blue (snake) sequence is exact.
  \begin{figure}
    \caption{The universal instance of the snake lemma}
  \label{fig:universal_instance_snake}
  \begin{center}
    \begin{tikzpicture}[label/.style={postaction={
      decorate,
      decoration={markings, mark=at position .5 with \node #1;}}},
      mylabel/.style={thick, draw=black, align=center, minimum width=0.5cm, minimum height=0.5cm,fill=white}]
          \coordinate (r) at (4.3,0);
          \coordinate (d) at (0,-3);
          
          {
          \node (V1) {$\EmbAdel{a}$};
          \node (V2) at ($(V1) + (r)$) {$\EmbAdel{b}$};
          \node (V3) at ($(V2) + (d)$) {$\EmbAdel{c}$};
          \node (V4) at ($(V3) + (r)$) {$\EmbAdel{d}$};
          
          \draw[->,thick] (V1) to node[above]{$\MorAdel{\alpha}$} (V2);
          \draw[->,thick] (V2) to node[right]{$\MorAdel{\beta}$} (V3);
          \draw[->,thick] (V3) to node[above]{$\MorAdel{\gamma}$} (V4);
          }

          {
          \node (Coka) at ($(V2) + (r)$) {$\ObjAdel{a}{\alpha}{b}{}{0}$};
          \draw[->,thick] (V2) to node[above]{$\MorAdel{\id_{b}}$} (Coka);
          }
          
          {
          \draw[->,thick] (Coka) to node[right]{$\MorAdel{\beta \cdot \gamma}$} (V4);
          }
          
          {
          \node (CokaZ) at ($(Coka) + 0.65*(r)$) {$0$};
          \draw[->,thick] (Coka) to (CokaZ);
          }

          {
          \node (K) at ($(Coka) - (d)$) {$\ObjAdel{a}{\alpha}{b}{\beta \cdot \gamma}{d}$};
          \draw[->,thick] (K) to node[right]{$\MorAdel{\id_{b}}$} (Coka);
          }

          {
          \node (Kerc) at ($(V3) - (r)$) {$\ObjAdel{0}{}{c}{\gamma}{d}$};
          \draw[->,thick] (Kerc) to node[above]{$\MorAdel{\id_{c}}$} (V3);
          }
          
          {
          \node (KercZ) at ($(Kerc) - 0.65*(r)$) {$0$};
          \draw[->,thick] (KercZ) to (Kerc);
          }

          {
          \draw[->,thick] (V1) to node[left]{$\MorAdel{\alpha \cdot \beta}$} (Kerc);
          }

          {
          \node (C) at ($(Kerc) + (d)$) {$\ObjAdel{a}{\alpha \cdot \beta}{c}{\gamma}{d}$};
          \draw[->,thick] (Kerc) to node[left]{$\MorAdel{\id_{c}}$} (C);
          }
          
          {
          \node (Kerb) at ($(V2) - (d)$) {$\ObjAdel{0}{}{b}{\beta}{c}$};
          \draw[->,thick,color = \snakecolor] (Kerb) to node[above]{$\MorAdel{\id_b}$} (K);
          \draw[->,thick] (Kerb) to node[right]{$\MorAdel{\id_b}$} (V2);
          
          \node (Kerab) at ($(V1) - (d)$) {$\ObjAdel{0}{}{a}{\alpha \cdot \beta}{c}$};
          \draw[->,thick] (Kerab) to node[left]{$\MorAdel{\id_a}$} (V1);
          \draw[->,thick,color = \snakecolor] (Kerab) to node[above]{$\MorAdel{\alpha}$} (Kerb);
          }
          
          {
          \node (Cokb) at ($(V3) + (d)$) {$\ObjAdel{b}{\beta}{c}{}{0}$};
          \draw[->,thick,color = \snakecolor] (C) to node[above]{$\MorAdel{\id_{c}}$} (Cokb);
          \draw[->,thick] (V3) to node[right]{$\MorAdel{\id_c}$} (Cokb);
          
          \node (Cokbc) at ($(V4) + (d)$) {$\ObjAdel{b}{\beta \cdot \gamma}{d}{}{0}$};
          \draw[->,thick,color = \snakecolor] (Cokb) to node[above]{$\MorAdel{\gamma}$} (Cokbc);
          \draw[->,thick] (V4) to node[right]{$\MorAdel{\id_{d}}$} (Cokbc);
          }
          
          {
            \node (KerabZ) at ($(Kerab) - 0.65*(d)$) {$0$};
            \draw[->,thick] (KerabZ) to (Kerab);
            
            \node (KerbZ) at ($(Kerb) - 0.65*(d)$) {$0$};
            \draw[->,thick] (KerbZ) to (Kerb);
            
            \node (Kzero) at ($(K) - 0.65*(d)$) {$0$};
            \draw[->,thick] (Kzero) to (K);
          }
          
          {
            \node (CokbcZ) at ($(Cokbc) + 0.65*(d)$) {$0$};
            \draw[->,thick] (Cokbc) to (CokbcZ);
            
            \node (CokbZ) at ($(Cokb) + 0.65*(d)$) {$0$};
            \draw[->,thick] (Cokb) to (CokbZ);
            
            \node (Czero) at ($(C) + 0.65*(d)$) {$0$};
            \draw[->,thick] (C) to (Czero);
          }

          \draw[->,thick,rounded corners,color = \snakecolor] ($(K.east)$) -| node[mylabel]{$\MorAdel{\beta}$} ($(CokaZ.east) + 0.3*(d) + 0.1*(r)$) -| ($(KercZ.west) - 0.1*(r)$) |-  ($(C.west)$);
    \end{tikzpicture}
    \end{center}
    \end{figure}
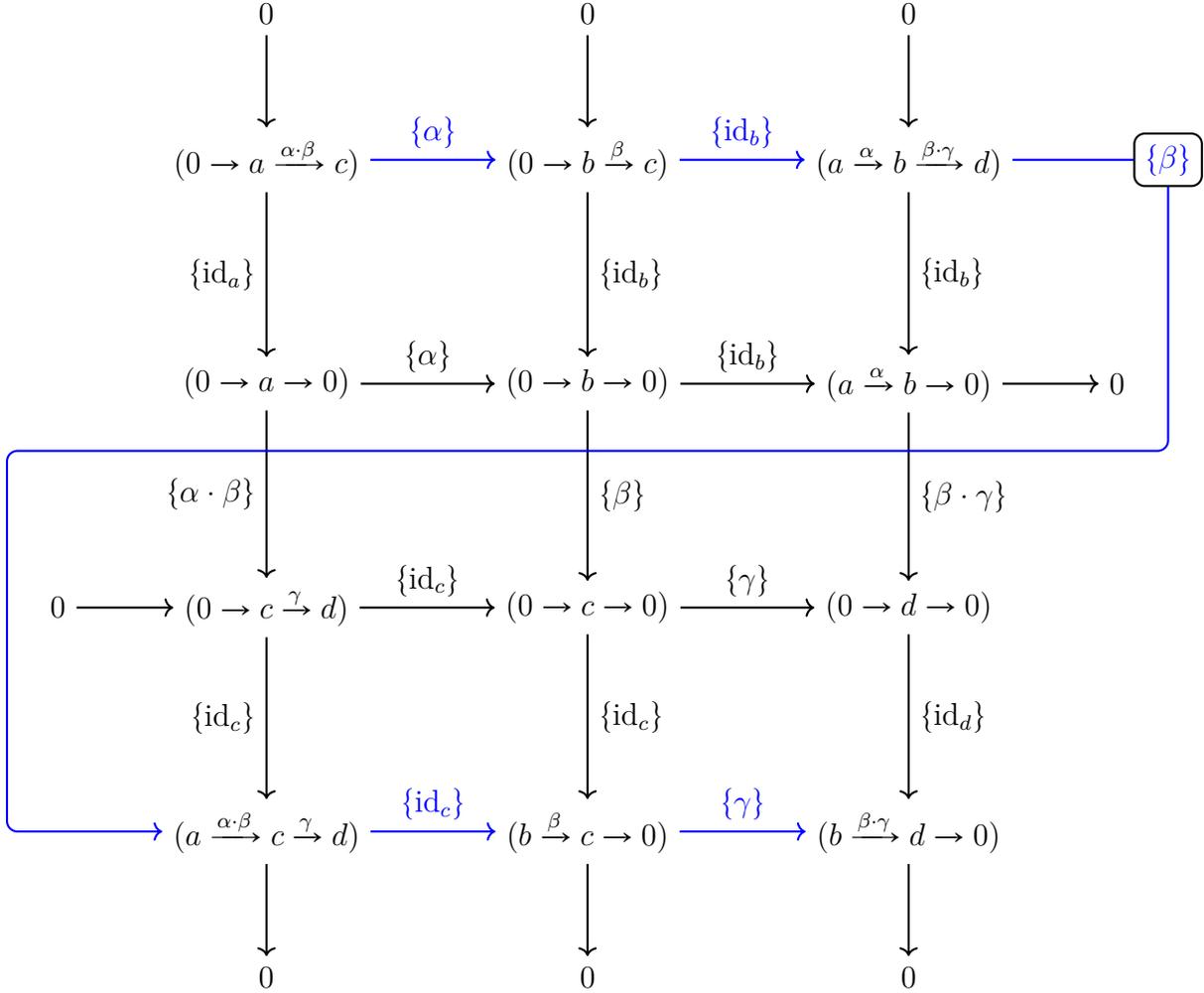
\end{lemma}
\begin{proof}
  All morphisms are trivially seen to be well-defined, and all squares commute already on the level of morphism data.
  The columns and the two rows in the middle are exact since we simply applied Construction \ref{construction:cokernels}
  and Construction \ref{construction:kernels}
  for the creation of cokernels and kernels. There are obvious witnesses for the composition of two consecutive morphisms being zero in the blue sequence.
  The only non-trivial computations are the verifications of exactness of the blue sequence.
  For example,
  in order to show that the blue sequence is exact at the source of the connecting homomorphism,
  we need to perform Computation \ref{computation:test_exactness}
  for $s = 1$.
  The other verifications are performed by similar computations.
\end{proof}

\begin{computation}\label{computation:test_exactness}
  For a parameter $s \in \Z$, we wish to see whether the sequence of the top row in the following diagram has zero homology:
  \begin{center}
    \begin{tikzpicture}[label/.style={postaction={
         decorate,
         decoration={markings, mark=at position .5 with \node #1;}},
         mylabel/.style={thick, draw=none, align=center, minimum width=0.5cm, minimum height=0.5cm,fill=white}},scale = 0.8]
         \coordinate (r) at (5,0);
         \coordinate (d) at (0,-3);
         \node (A) {$\ObjAdel{0}{}{b}{\beta}{c}$};
         \node (B) at ($(A)+(r)$) {$\ObjAdel{a}{\alpha}{b}{\beta \cdot \gamma}{d}$};
         \node (C) at ($(B) + (r)$) {$\ObjAdel{a}{\alpha \cdot \beta}{c}{\gamma}{d}$};
         \node (ker) at ($(A) + (d) + 0.1*(r)$) {$\ObjAdel{a^{\oplus 2}}{\pmattwobytwos{\alpha}{0}{0}{\id_a}}{b \oplus a}{\pmattwobytwos{\beta \cdot \gamma}{s\beta}{0}{\alpha \cdot \beta}}{d \oplus c}$};
         \node (coker) at ($(ker) + 2*(r)$) {$\ObjAdel{a \oplus b}{\pmattwobytwos{\alpha}{0}{\id_b}{\beta}}{b \oplus c}{\pmattwobytwos{\beta \cdot \gamma}{0}{0}{\id_c}}{d \oplus c}$};
         
         \draw[->,thick] (A) to node[above]{$\MorAdel{\id_b}$} (B);
         \draw[->,thick] (B) to node[above]{$\MorAdel{s\beta}$}(C);
         \draw[right hook->,thick] (ker) to (B);
         \draw[->,thick] ($(ker.east) + 0.15*(d)$) to node[above,yshift=0.5em]{$\pmattwobytwos{id_b}{0}{0}{0}$} ($(coker.west) + 0.15*(d)$);
         \draw[->>,thick] (B) to (coker);
   \end{tikzpicture}
\end{center}
But this amounts to showing that the morphism in the bottom row, which is the composition of the kernel embedding
of $\MorAdel{s\beta}$ with the cokernel projection of $\MorAdel{\id_b}$ of that diagram is the zero morphism.
It easily follows that this is only the case for $s \in \{-1,1\}$.
For these cases, we depict a witness pair for being zero in the following diagram:
\begin{center}
  \begin{tikzpicture}[label/.style={postaction={
          decorate,
          decoration={markings, mark=at position .5 with \node #1;}}}, baseline = (A),
          mylabel/.style={thick, draw=none, align=center, minimum width=0.5cm, minimum height=0.5cm,fill=white}]
        \coordinate (r) at (4,0);
        \coordinate (u) at (0,-3);
        
        \node (RB) {};
        \node (B) at ($(RB) + (r)$) {$b \oplus a$};
        \node (CB) at ($(B) + (r)$) {$d \oplus c$};
        
        \node (X) at ($(RB) + (u)$) {$a \oplus b$};
        \node (Y) at ($(X) + (r)$) {$b \oplus c$};

        \draw[->,thick] (B) --node[above]{$\pmattwobytwos{\beta \cdot \gamma}{s\beta}{0}{\alpha \cdot \beta}$} (CB);
        
        \draw[->,thick] (X) --node[above]{$\pmattwobytwos{\alpha}{0}{\id_b}{\beta}$} (Y);

        \draw[->,thick] (B) --node[mylabel]{$\pmattwobytwos{\id_b}{0}{0}{0}$} (Y);
        
        \draw[->,thick, dashed, out = 180, in = 90] (B) to node[mylabel]
        {$\pmattwobytwos{0}{\id_b}{-s\id_a}{s\alpha}$}
        (X);
        
        \draw[->,thick, dashed, out = -90, in = 0] (CB) to node[mylabel]
        {$\pmattwobytwos{0}{0}{0}{-s\id_c}$}
        (Y);
        
  \end{tikzpicture}
\end{center}
\end{computation}

\begin{corollary}
  The snake lemma holds in every abelian category.
\end{corollary}
\begin{proof}
  Clear from the universal property of Adelman categories.
  For the sake of clarity, we nevertheless provide a more detailed explanation.
  Let $\BC$ be an abelian category with morphisms $\alpha$, $\beta$, $\gamma$
  as in the premise of Lemma \ref{lemma:snake_rephrased}.
  Let $F: \AsCat( Q, R )^{\oplus} \rightarrow \BC$ be the additive functor that is determined (up to natural isomorphism)
  by sending $\alpha$, $\beta$, $\gamma$ in $Q \subseteq \AsCat( Q, R )^{\oplus}$
  to $\alpha$, $\beta$, $\gamma$ in $\BC$.
  Then, by the universal property of Adelman categories, we obtain an exact functor 
  $\widehat{F}: \Adel( \AsCat( Q, R )^{\oplus} ) \rightarrow \BC$.
  Due to its exactness, $\widehat{F}$ sends the universal instance depicted in Figure \ref{fig:universal_instance_snake}
  to the diagram and the snake sequence depicted in Lemma \ref{lemma:snake_rephrased}.
  Thus, the claim holds.
\end{proof}

\begin{remark}[Dowker's explicit formula]\label{remark:concrete_formula}
  The connecting homomorphism in the universal instance together with Remark \ref{remark:explicit_up_induced_fun}
  yields a concrete formula for the connecting homomorphism in any abelian category. It is given by
  \[
    \partial = \CH\left( \ObjAdel{a}{\alpha}{b}{\beta \cdot \gamma}{d} \xrightarrow{\MorAdel{\beta}} \ObjAdel{a}{\alpha \cdot \beta}{c}{\gamma}{d} \right)
  \]
  This construction goes back to Dowker in \cite[Theorem 3]{Dow66}.
  In particular, we see that $\kernel( \epsilon ) \cong \CH\ObjAdel{a}{\alpha}{b}{\beta \cdot \gamma}{d}$
  and $\cokernel( \delta ) \cong \CH\ObjAdel{a}{\alpha \cdot \beta}{c}{\gamma}{d}$.
\end{remark}

\begin{lemma}[Universal uniqueness of the connecting homomorphism]\label{lemma:universal_uniqueness}
  Let $Q$ and $R$ be as in Lemma \ref{lemma:universal_snake_lemma}.
  Then
  \[
    \Hom_{\Adel( \AsCat( Q, R )^{\oplus} )}\left( \ObjAdel{a}{\alpha}{b}{\beta \cdot \gamma}{d}, \ObjAdel{a}{\alpha \cdot \beta}{c}{\gamma}{d} \right) \simeq \Z
  \]
  with generator given by $\MorAdel{\beta}$.
  Moreover, only this generator or its additive inverse render the blue sequence of
  Lemma \ref{lemma:universal_snake_lemma} exact.
\end{lemma}
\begin{proof}
  Clearly, $\MorAdel{\beta}$ generates this homomorphism set. Moreover, there cannot be torsion since
  $\Hom( b, a ) \simeq 0$ and $\Hom( d, c ) \simeq 0$.
  The last assertion follows from Computation \ref{computation:test_exactness} for all cases $s \in \Z$.
\end{proof}

\begin{remark}\label{remark:interpret_universal_uniqueness}
We may interpret Lemma \ref{lemma:universal_uniqueness} as follows:
every ``construction recipe'' of the connecting homomorphism in the snake lemma that works for all abelian categories and
that only uses categorical constructions that are preserved by exact functors (like taking kernels and cokernels)
yields the same morphism $\kernel( \epsilon ) \rightarrow \cokernel( \delta )$ up to isomorphism,
i.e., up to precomposition with an automorphism of $\kernel( \epsilon )$ and postcomposition with an automorphism of $\cokernel( \delta )$.
\end{remark}

\begin{example}\label{example:direct_examples}
The construction in \cite[Chapter VIII, Lemma 5]{MLCWM} has to yield a result equal to
Dowker's concrete formula in Remark \ref{remark:concrete_formula} up to isomorphism.
Moreover, the construction using generalized morphisms (i.e., relations instead of homomorphisms \cite[Chapter 2, Lemma 2.1]{PosurDoktor})
also has to yield the same result.
\end{example}

\begin{example}[Freyd-Mitchell embedding]\label{example:embedding}
  A classical argument for the existence of the connecting homomorphism uses the Freyd-Mitchell embedding theorem \cite{Freyd}:
  every small abelian category $\BC$ admits a full embedding into a module category $R\Modl$ for some ring $R$.
  Within $R\Modl$ the connecting homomorphism $\partial$ can be constructed by the usual chasing of elements \cite{weihom} (and is afterwards pulled back to $\BC$ along the full embedding).
  But now, it is not hard to see that this particular $\partial$ could also be obtained (up to isomorphism) by
  any of the constructions in Example \ref{example:direct_examples} in the context of $R\Modl$ (and thus, due to the full embedding, in the context of $\BC$).
  Thus, the classical argument based on the Freyd-Mitchell embedding yields the same morphism up to isomorphism as all the previously mentioned constructions.
\end{example}

Note that the statement of universal uniqueness does not claim 
that the connecting homomorphism is uniquely determined up to isomorphism if we focus only on a particular instance.

\begin{example}
 In order to create a counterexample for the uniqueness of the connecting homomorphism in a particular instance, it suffices
 to given an example of two non-isomorphic morphisms in an abelian category with the same kernel and the same cokernel.
 Let $E \coloneq \bigwedge\langle e_0, e_1, e_2 \rangle$ be the $\Z$-graded exterior algebra over $\Q$ in $3$ variables with $\deg( e_i ) = 1$ for $i = 0,1,2$.
 We take the quotient algebra $R \coloneq E/\langle e_0 \wedge e_1 \wedge e_2 \rangle$.
 By abuse of notation, we also refer to the residue classes of elements in $R$ by $e_0, e_1, e_2$.

 In this example, we work in the category of $\Z$-graded left modules over $R$.
 We define the submodule
 \[
   M \coloneq \langle e_0, e_1, e_2 \rangle \hookrightarrow R
 \]
 and its graded shift $M(1)$ with graded parts $M(1)_i \coloneq M_{i+1}$ for $i \in \Z$.
 We have 
 \[ \Hom_E( M, M(1) ) \cong \Q^{3 \times 3}\]
 since we may map each $e_i \in M$ to any linear combination of the monomials 
 $e_0 \wedge e_1$, $e_0 \wedge e_2$, $e_1 \wedge e_2$ in $M(1)$.
 If $\beta \in \Hom_E( M, M(1) )$ corresponds to an invertible matrix in $\Q^{3 \times 3}$,
 then
 \[
   \kernel( \beta ) = \langle e_0 \wedge e_1, e_0 \wedge e_2, e_1 \wedge e_2 \rangle
 \]
 and
 \[
   \cokernel( \beta ) = \frac{M}{\langle e_0 \wedge e_1, e_0 \wedge e_2, e_1 \wedge e_2 \rangle}(1).
 \]
 It follows that if $\beta \in \Hom_E( M, M(1) )$ corresponds to the identity $3\times 3$ matrix,
 and if $\beta' \in \Hom_E( M, M(1) )$ corresponds any non-trivial permutation matrix, then the following holds:
 \begin{enumerate}
   \item The sequences
   \[
    0 \rightarrow \kernel(\beta) \rightarrow M \xrightarrow{\beta} M(1) \rightarrow \cokernel( \beta ) \rightarrow 0
  \]
  and
  \[
    0 \rightarrow \kernel(\beta) \rightarrow M \xrightarrow{\beta'} M(1) \rightarrow \cokernel( \beta ) \rightarrow 0
  \]
  are exact, since $\beta$ and $\beta'$ have equal kernels and cokernels.
  \item There are no automorphisms $\sigma \in \Aut_E(M)$ and $\tau \in \Aut_E( M(1) )$ such that
  \begin{center}
  \begin{tikzpicture}[label/.style={postaction={
    decorate,
    decoration={markings, mark=at position .5 with \node #1;}}}, baseline = (A),
    mylabel/.style={thick, draw=none, align=center, minimum width=0.5cm, minimum height=0.5cm,fill=white}]
  \coordinate (r) at (4,0);
  \coordinate (d) at (0,1.5);
  
  \node (A) {$M$};
  \node (B) at ($(A) + (r)$) {$M(1)$};
  \node (C) at ($(A) + (d)$) {$M$};
  \node (D) at ($(C) + (r)$) {$M(1)$};
  
  \draw[->,thick] (A) --node[above]{$\beta$} (B);
  \draw[->,thick] (C) --node[left]{$\sigma$} (A);
  \draw[->,thick] (D) --node[right]{$\tau$} (B);
  \draw[->,thick] (C) --node[above]{$\beta'$} (D);
  \end{tikzpicture}
  \end{center}
  commutes, since such automorphisms can only send each $e_i$ to a scalar multiple.
 \end{enumerate}
 It follows that for the triple of morphisms
 \[
   0 \rightarrow M \xrightarrow{\beta} M(1) \rightarrow 0
 \]
 in the premise of the snake lemma \ref{lemma:snake_rephrased}, there are non-isomorphic connecting homomorphisms that satisfy the conclusion.
\end{example}

\subsection{The universal instance of a refinement of the 5-lemma}\label{subsection:five_lemma}
In this subsection, we discuss the following version of the 5-lemma.
\begin{lemma}[5-lemma]\label{lemma:five}
  Suppose given the following commutative diagram in an abelian category:
  \begin{center}
    \begin{tikzpicture}[label/.style={postaction={
            decorate,
            decoration={markings, mark=at position .5 with \node #1;}}}, baseline = (A),
            mylabel/.style={thick, draw=none, align=center, minimum width=0.5cm, minimum height=0.5cm,fill=white}]
          \coordinate (r) at (4,0);
          \coordinate (d) at (0,-1.5);
          
          \node (A) {$a$};
          \node (B) at ($(A) + (r)$) {$b$};
          \node (C) at ($(B) + (r)$) {$c$};
          \node (D) at ($(C) + (r)$) {$d$};
          
          \node (E) at ($(A) + (d)$){$e$};
          \node (F) at ($(B) + (d)$) {$f$};
          \node (G) at ($(C) + (d)$) {$g$};
          \node (H) at ($(D) + (d)$) {$h$};
          
          \draw[->,thick] (A) --node[above]{$\alpha$} (B);
          \draw[->,thick] (B) --node[above]{$\beta$} (C);
          \draw[->,thick] (C) --node[above]{$\gamma$} (D);
          
          \draw[->>,thick] (A) --node[left]{$\delta$} (E);
          \draw[right hook->,thick] (B) --node[left]{$\epsilon$} (F);
          \draw[->,thick] (C) --node[left]{$\zeta$} (G);
          \draw[right hook->,thick] (D) --node[left]{$\eta$} (H);
          
          \draw[->,thick] (E) --node[above]{$\theta$} (F);
          \draw[->,thick] (F) --node[above]{$\iota$} (G);
          \draw[->,thick] (G) --node[above]{$\kappa$} (H);
          
    \end{tikzpicture}
  \end{center}
  If the composition of any two consecutive horizontal morphisms is zero, 
  $\delta$ is an epi, $\eta$ and $\epsilon$ are monos,
  and if we have exactness at $c$ and $f$,
  then $\zeta$ is a mono.
\end{lemma}

Since it is not obvious how to encode the premise of the $5$-lemma as an additive functor,
the first challenge is to refine the lemma in a way such that such an encoding becomes possible.
In contrast to the snake lemma, we will not only refine the premise, but also the conclusion of the $5$-lemma.

\begin{lemma}[Refined 5-lemma]\label{lemma:five_refined}
  Suppose given the following commutative diagram in an abelian category:
  \begin{center}
    \begin{tikzpicture}[label/.style={postaction={
            decorate,
            decoration={markings, mark=at position .5 with \node #1;}}}, baseline = (A),
            mylabel/.style={thick, draw=none, align=center, minimum width=0.5cm, minimum height=0.5cm,fill=white}]
          \coordinate (r) at (4,0);
          \coordinate (d) at (0,-1.5);
          
          \node (A) {$a$};
          \node (B) at ($(A) + (r)$) {$b$};
          \node (C) at ($(B) + (r)$) {$c$};
          \node (D) at ($(C) + (r)$) {$d$};
          
          \node (E) at ($(A) + (d)$){$e$};
          \node (F) at ($(B) + (d)$) {$f$};
          \node (G) at ($(C) + (d)$) {$g$};
          \node (H) at ($(D) + (d)$) {$h$};
          
          \draw[->,thick] (A) --node[above]{$\alpha$} (B);
          \draw[->,thick] (B) --node[above]{$\beta$} (C);
          \draw[->,thick] (C) --node[above]{$\gamma$} (D);
          
          \draw[->>,thick] (A) --node[left]{$\delta$} (E);
          \draw[->,thick] (B) --node[left]{$\epsilon$} (F);
          \draw[->,thick] (C) --node[left]{$\zeta$} (G);
          \draw[right hook->,thick] (D) --node[left]{$\eta$} (H);
          
          \draw[->,thick] (E) --node[above]{$\theta$} (F);
          \draw[->,thick] (F) --node[above]{$\iota$} (G);
          \draw[->,thick] (G) --node[above]{$\kappa$} (H);
          
    \end{tikzpicture}
  \end{center}
  If the composition of any two consecutive horizontal morphisms is zero, $\delta$ is an epi, and $\eta$ is a mono,
  then we have a monomorphism
  \[
    \CH\big(\kernel( \epsilon ) \rightarrow \kernel( \zeta ) \rightarrow \CH( \beta, \gamma ) \big)
    \hookrightarrow
    \cokernel\big( \CH( {\alpha}, {\beta} ) \xrightarrow{\CH( \epsilon ) } \CH( {\theta}, {\iota} ) \big).
  \]
  In particular, $\kernel( \zeta )$ admits a filtration whose graded parts are given by subquotients of the three objects
  \[
    \kernel( \epsilon ), \CH( \beta, \gamma ), \CH( \theta, \iota ).
  \]
  In particular, if these three objects are zero (i.e., if $\epsilon$ is monic, the diagram is exact at $c$ and $f$), then $\zeta$ is a mono.
\end{lemma}

Note that, since $\delta$ is an epimorphism, it has to be the cokernel of some morphism $\lambda$.
Similarly, since $\eta$ is a monomorphism, it has to be the kernel of some morphism $\mu$.
These two observation suffice in order to encode the premise of Lemma \ref{lemma:five_refined} as an additive functor, as we will see in the next lemma.

\begin{lemma}[Universal instance of the refined 5-lemma]\label{lemma:universal_instace_5lem}
  Let $Q$ denote the quiver given by
  \begin{center}
    \begin{tikzpicture}[label/.style={postaction={
      decorate,
      decoration={markings, mark=at position .5 with \node #1;}},
      mylabel/.style={thick, draw=none, align=center, minimum width=0.5cm, minimum height=0.5cm,fill=white}}, scale = 1]
          \coordinate (r) at (2,0);
          \coordinate (d) at (0,-2);
          
          {
          \node (I) {$i$};
          \node (A) at ($(I) + (d)$) {$a$};
          \node (B) at ($(A) + (r)$) {$b$};
          \node (C) at ($(B) + (r)$) {$c$};
          
          \node (F) at ($(B) + (d)$) {$f$};
          \node (G) at ($(F) + (r)$) {$g$};
          \node (H) at ($(G) + (r)$) {$h$};
          \node (J) at ($(H) + (d)$) {$j$};
          
          \draw[->,thick] (I) to node[left]{$\lambda$} (A);

          \draw[->,thick] (A) to node[above]{$\alpha$} (B);
          \draw[->,thick] (B) to node[above]{$\beta$} (C);
          
          \draw[->,thick] (B) to node[left]{$\epsilon$} (F);
          \draw[->,thick] (C) to node[left]{$\zeta$} (G);
          
          \draw[->,thick] (F) to node[above]{$\iota$} (G);
          \draw[->,thick] (G) to node[above]{$\kappa$} (H);
          
          \draw[->,thick] (H) to node[right]{$\mu$} (J);
          }
          
    \end{tikzpicture}
  \end{center}
  and let $R$ be given by the relations
  \[
    \alpha \cdot \beta = 0, \hspace{1em}\iota \cdot \kappa = 0, \hspace{1em}\beta \cdot \zeta = \epsilon \cdot \iota.
  \]
  Then we get a commutative diagram in $\Adel( \AsCat( Q, R )^{\oplus} )$
  \begin{center}
    \begin{tikzpicture}[label/.style={postaction={
            decorate,
            decoration={markings, mark=at position .5 with \node #1;}}}, baseline = (A),
            mylabel/.style={thick, draw=none, align=center, minimum width=0.5cm, minimum height=0.5cm,fill=white}]
          \coordinate (r) at (4,0);
          \coordinate (u) at (0,-3);
          
          \node (A) {$\EmbAdel{a}$};
          \node (B) at ($(A) + (r)$) {$\EmbAdel{b}$};
          \node (C) at ($(B) + (r)$) {$\EmbAdel{c}$};
          \node (D) at ($(C) + (r)$) {$\ObjAdel{0}{}{h}{\mu}{j}$};
          
          \node (E) at ($(A) + (d)$){$\ObjAdel{i}{\lambda}{a}{}{0}$};
          \node (F) at ($(B) + (d)$) {$\EmbAdel{f}$};
          \node (G) at ($(C) + (d)$) {$\EmbAdel{g}$};
          \node (H) at ($(D) + (d)$) {$\EmbAdel{h}$};
          
          \draw[->,thick] (A) --node[above]{$\MorAdel{\alpha}$} (B);
          \draw[->,thick] (B) --node[above]{$\MorAdel{\beta}$} (C);
          \draw[->,thick] (C) --node[above]{$\MorAdel{\zeta \cdot \kappa}$} (D);
          
          \draw[->>,thick] (A) --node[left]{$\delta := \MorAdel{\id_a}$} (E);
          \draw[->,thick] (B) --node[left]{$\MorAdel{\epsilon}$} (F);
          \draw[->,thick] (C) --node[left]{$\MorAdel{\zeta}$} (G);
          \draw[right hook->,thick] (D) --node[right]{$\MorAdel{\id_h} =:\eta$} (H);
          
          \draw[->,thick] (E) --node[above]{$\MorAdel{\alpha \cdot \epsilon}$} (F);
          \draw[->,thick] (F) --node[above]{$\MorAdel{\iota}$} (G);
          \draw[->,thick] (G) --node[above]{$\MorAdel{\kappa}$} (H);
          
    \end{tikzpicture}
  \end{center}
  that satisfies the premise and the conclusion of Lemma \ref{lemma:five_refined}.
\end{lemma}
\begin{proof}
  The outer left morphism is an epi since it is constructed as a cokernel projection (see Construction \ref{construction:cokernels}),
  dually, the outer right morphism is a mono.
  In order reach the conclusion of Lemma \ref{lemma:five_refined},
  it suffices to show the following facts step by step using basic computations within $\Adel( \AsCat( Q, R )^{\oplus} )$:
  \begin{enumerate}
    \item We have an isomorphism $\CH( \MorAdel{\beta}, \MorAdel{\zeta \cdot \kappa }) \simeq \ObjAdel{b}{\beta}{c}{\zeta \cdot \kappa}{h}$.
    \item  We have a well-defined sequence
    \[
      \ObjAdel{0}{}{b}{\epsilon}{f} \xrightarrow{\MorAdel{\beta}} \ObjAdel{0}{}{c}{\zeta}{g} \xrightarrow{\MorAdel{\id_c}} \ObjAdel{b}{\beta}{c}{\zeta \cdot \kappa}{h}
    \]
    whose homology is given by the object
    \[
      \ObjAdel{b\oplus b}{\pmattwobythree{\beta}{\epsilon}{0}{0}{0}{\id_b}}{c \oplus f \oplus b}{\pmatthreebythree{\zeta}{0}{\id_c}{0}{\id_f}{0}{0}{0}{\beta}}{g \oplus f \oplus c}
    \]
    \item We have an isomorphism
    \[
      \cokernel\big( \CH( \MorAdel{\alpha}, \MorAdel{\beta} ) \xrightarrow{\CH( \MorAdel{\epsilon} ) } \CH( \MorAdel{\theta}, \MorAdel{\iota} ) \big)
      \simeq
      \ObjAdel{a \oplus b}{\pmattwobytwo{\alpha \cdot \epsilon}{0}{\epsilon}{\beta}}{f \oplus c}{ \pmattwobytwo{\iota}{0}{0}{\id_c}}{g \oplus c}
    \]
    \item We have a monomorphism from the homology object of step $(2)$ to the object in step $(3)$:
    \begin{center}
      \begin{tikzpicture}[label/.style={postaction={
              decorate,
              decoration={markings, mark=at position .5 with \node #1;}}}, baseline = (A),
              mylabel/.style={thick, draw=none, align=center, minimum width=0.5cm, minimum height=0.5cm,fill=white}]
            \coordinate (r) at (5,0);
            \coordinate (d) at (0,-2.5);
            
            \node (A) {$\big( b \oplus b$};
            \node (B) at ($(A) + (r)$) {$c \oplus f \oplus b$};
            \node (C) at ($(B) + (r)$) {$g \oplus f \oplus c \big)$};
            
            \node (E) at ($(A) + (d)$){$\big( a \oplus b$};
            \node (F) at ($(B) + (d)$) {$f \oplus c$};
            \node (G) at ($(C) + (d)$) {$g \oplus c \big)$};
            
            \draw[->,thick] (A) --node[above]{$\pmattwobythree{\beta}{\epsilon}{0}{0}{0}{\id_b}$} (B);
            \draw[->,thick] (B) --node[above]{$\pmatthreebythree{\zeta}{0}{\id_c}{0}{\id_f}{0}{0}{0}{\beta}$} (C);
            
            \draw[->,dashed,thick] (A) --node[left]{$\pmattwobytwo{0}{\id_b}{0}{\id_b}$} (E);
            \draw[->,thick] (B) --node[left]{$\pmatthreebytwo{0}{\id_c}{\id_f}{0}{\epsilon}{\beta}$} (F);
            \draw[->,dashed,thick] (C) --node[right]{$\pmatthreebytwo{-\id_g}{0}{\iota}{0}{\zeta}{\id_c}$} (G);

            \draw[->,thick] (E) --node[below]{$\pmattwobytwo{\alpha \cdot \epsilon}{0}{\epsilon}{\beta}$} (F);
            \draw[->,thick] (F) --node[below]{$\pmattwobytwo{\iota}{0}{0}{\id_c}$} (G);
            
      \end{tikzpicture}
    \end{center}
  \end{enumerate}
   
  For $(1)$, the claim follows from $\CH( \MorAdel{\beta}, \MorAdel{\zeta \cdot \kappa }) \simeq \CH( \MorAdel{\beta}, \MorAdel{\zeta \cdot \kappa } \cdot \eta)$.
  For $(2)$, first we factor $\MorAdel{\id_c}$ via the cokernel projection of $\MorAdel{\beta}$ and a uniquely determined morphism $\nu$, and second we take the kernel of $\nu$
  in order to obtain our desired homology object. For $(3)$, we identify $\CH( \MorAdel{\epsilon})$ with
  $
    \ObjAdel{a}{\alpha}{b}{\beta}{c} \xrightarrow[]{\MorAdel{\epsilon}} \ObjAdel{a}{\alpha \cdot \epsilon}{f}{\iota}{g}
  $
  and compute its cokernel.
  For $(4)$, we check if the identity morphism of the kernel object of the depicted morphism is zero,
  and a witness pair of that fact is explicitly given by the following diagram:
  \begin{center}
    \begin{tikzpicture}[label/.style={postaction={
            decorate,
            decoration={markings, mark=at position .5 with \node #1;}}}, baseline = (A),
            mylabel/.style={thick, draw=none, align=center, minimum width=0.5cm, minimum height=0.5cm,fill=white}]
          \coordinate (r) at (6.5,0);
          \coordinate (u) at (0,-4);
          
          \node (RB) {};
          \node (B) at ($(RB) + (r)$) {$c \oplus f \oplus b \oplus a \oplus b$};
          \node (CB) at ($(B) + (r)$) {$g \oplus f \oplus c \oplus f \oplus c$};
          
          \node (X) at ($(RB) + (u)$) {$b \oplus b \oplus a \oplus b$};
          \node (Y) at ($(X) + (r)$) {$c \oplus f \oplus b \oplus a \oplus b$};

          \draw[->,thick] (B) --node[above,yshift=0.4em]{
            $\begin{pmatrix}\zeta & 0 & \id_c & 0 & \id_c \\ 0 & \id_f & 0 & \id_f & 0 \\ 0 & 0 & \beta & \epsilon & \beta \\ 0 & 0 & 0 & \alpha \cdot \epsilon & 0 \\ 0 & 0 & 0 & \epsilon & \beta \end{pmatrix}$
            } (CB);
          
          \draw[->,thick] (X) --node[below,yshift=-0.4em]{
            $
            \begin{pmatrix} \beta & \epsilon & 0 & 0 & 0 \\ 0 & 0 & \id_b & 0 & 0 \\ 0 & 0 & 0 & \id_a & 0 \\ 0 & 0 & 0 & 0 & \id_b \end{pmatrix}
            $} (Y);

          \draw[->,thick] (B) --node[mylabel]{$\id$} (Y);
          
          \draw[->,thick, dashed, out = 180, in = 90] (B) to node[mylabel]
          {
            $
            \begin{pmatrix} 
            0 & 0 & 0 & 0 \\
            0 & 0 & 0 & 0 \\
            -\id_b & \id_b & 0 & 0 \\
            -\alpha & 0 & \id_a & 0 \\
            -\id_b & 0 & 0 & \id_b
            \end{pmatrix}
            $
          }
          (X);
          
          \draw[->,thick, dashed, out = -90, in = 0] (CB) to node[mylabel]
          {
            $
            \begin{pmatrix}
              0 & 0 & 0 & 0 & 0 \\
              0 & 0 & 0 & 0 & 0 \\
              0 & 0 & 0 & 0 & 0 \\
              0 & \id_f & 0 & 0 & 0 \\
              \id_c & 0 & 0 & 0 & 0 \\
            \end{pmatrix}
            $}
          (Y);
          
    \end{tikzpicture}
  \end{center}
\end{proof}

\begin{corollary}
  The refined $5$-lemma holds in every abelian category. In particular, the $5$-lemma holds in every abelian category.
\end{corollary}

\section{Conclusion and outlook}\label{section:outlook}

Adelman categories are amenable to 
a computer implementation\footnote{Such an implementation based on \CapPkg \cite{CAP-project} can be found on the github account of the 
author: \href{https://github.com/sebastianpos/Adelman.jl}{\url{https://github.com/sebastianpos/Adelman.jl}} }.
Thus, our methods allow for experimentation in the spirit of computer algebra.
The chosen examples (snake lemma, 5-lemma) of this paper serve as an illustration of our methods,
and we observe that our methods can shed some new light even on classical lemmata.
We could also have illustrated our methods by
direct computations in $\Adel( \AsCat( D_4 )^{\oplus} )$ that allow to solve the $3$ subspace problem and to recover Dedekind's free modular lattice on $3$ 
generators\footnote{This is nicely depicted in 
\href{https://blogs.ams.org/visualinsight/2016/01/01/free-modular-lattice-on-3-generators/}{\url{https://blogs.ams.org/visualinsight/2016/01/01/free-modular-lattice-on-3-generators/}}
}, where $D_4$ is given by the Dynkin quiver
\begin{center}
  \begin{tikzpicture}[label/.style={postaction={
       decorate,
       decoration={markings, mark=at position .5 with \node #1;}},
       mylabel/.style={thick, draw=none, align=center, minimum width=0.5cm, minimum height=0.5cm,fill=white}}]
       \coordinate (r) at (1,0);
       \coordinate (u) at (0,1);
       \node (A) {$\bullet$};
       \node (B) at ($(A)+(u)$) {$\bullet$};
       \node (C) at ($(A) + (r) - (u)$) {$\bullet$};
       \node (D) at ($(A) - (r) - (u)$) {$\bullet$};
       \draw[->,thick] (B) to (A);
       \draw[->,thick] (C) to (A);
       \draw[->,thick] (D) to (A);
 \end{tikzpicture}
\end{center}

Moreover, Adelman categories can serve as a formulation of homological algebra without the usage of elements and embedding theorems:
once an elementary proof of the universal property of the Adelman category is given, the proofs in this paper can be regarded as elementary
proofs of lemmata that classically are proven by chasing elements.

Last, we saw in Subsection \ref{subsection:five_lemma} that it is difficult to apply our methods directly to the proof of the 5-lemma.
Instead, we first needed to find a refinement of the premise and the conclusion.
A natural way to avoid this difficulty is given by Serre quotients:
let $Q$ denote the quiver given by 
\begin{center}
  \begin{tikzpicture}[label/.style={postaction={
          decorate,
          decoration={markings, mark=at position .5 with \node #1;}}}, baseline = (A),
          mylabel/.style={thick, draw=none, align=center, minimum width=0.5cm, minimum height=0.5cm,fill=white}]
        \coordinate (r) at (4,0);
        \coordinate (d) at (0,-1.5);
        
        \node (A) {$a$};
        \node (B) at ($(A) + (r)$) {$b$};
        \node (C) at ($(B) + (r)$) {$c$};
        \node (D) at ($(C) + (r)$) {$d$};
        
        \node (E) at ($(A) + (d)$){$e$};
        \node (F) at ($(B) + (d)$) {$f$};
        \node (G) at ($(C) + (d)$) {$g$};
        \node (H) at ($(D) + (d)$) {$h$};
        
        \draw[->,thick] (A) --node[above]{$\alpha$} (B);
        \draw[->,thick] (B) --node[above]{$\beta$} (C);
        \draw[->,thick] (C) --node[above]{$\gamma$} (D);
        
        \draw[->,thick] (A) --node[left]{$\delta$} (E);
        \draw[->,thick] (B) --node[left]{$\epsilon$} (F);
        \draw[->,thick] (C) --node[left]{$\zeta$} (G);
        \draw[->,thick] (D) --node[left]{$\eta$} (H);
        
        \draw[->,thick] (E) --node[above]{$\theta$} (F);
        \draw[->,thick] (F) --node[above]{$\iota$} (G);
        \draw[->,thick] (G) --node[above]{$\kappa$} (H);
        
  \end{tikzpicture}
\end{center}
and let $R$ be relations which encode that any two consecutive horizontal morphisms compose to zero,
and that the three rectangles commute.
Then, we may find our desired universal instance of the five lemma within the Serre quotient category
\[
  \Adel( \AsCat( Q, R)^{\oplus } )/ \CC
\]
where $\CC$ is the Serre subcategory spanned by the objects
\begin{itemize}
  \item $\CH( \beta, \gamma )$ (exactness at $c$),
  \item $\CH( \theta, \iota )$ (exactness at $f$),
  \item $\cokernel( \delta )$ ($\delta$ is epic),
  \item $\kernel( \epsilon )$ ($\epsilon$ is monic),
  \item $\kernel( \eta )$ ($\eta$ is monic).
\end{itemize}
It follows that proving the $5$-lemma means checking if $\zeta$ is monic in $\Adel( \AsCat( Q, R)^{\oplus } )/ \CC$,
which means checking $\kernel( \zeta ) \in \CC$.
Motivated by this example, a constructive treatment of such Serre quotient categories appears to be desirable.

%% file: adelman.bbl
\def\cprime{$'$} \def\cprime{$'$} \def\cprime{$'$} \def\cprime{$'$}
  \def\cprime{$'$}
\providecommand{\bysame}{\leavevmode\hbox to3em{\hrulefill}\thinspace}
\providecommand{\MR}{\relax\ifhmode\unskip\space\fi MR }
\providecommand{\MRhref}[2]{%
  \href{http://www.ams.org/mathscinet-getitem?mr=#1}{#2}
}
\providecommand{\href}[2]{#2}